\documentclass[11pt,reqno]{amsart}
\usepackage{amssymb,latexsym}
\usepackage[colorlinks=true]{hyperref}
\usepackage[hmargin=3.2cm, vmargin=3.2cm]{geometry}
\usepackage{amsmath}
\usepackage{latexsym}
\usepackage{amssymb,amsthm,amsfonts}
\usepackage{graphicx}
\usepackage{color} 

\makeatletter
\def\@currentlabel{2.1}\label{e:dispaa}
\def\@currentlabel{2.21}\label{e:dispau}
\def\@currentlabel{2.22}\label{e:dispav}
\def\@currentlabel{2.23}\label{e:dispaw}
\def\@currentlabel{2.24}\label{e:dispax}
\def\theequation{\thesection.\@arabic\c@equation}
\makeatother

\newcommand{\D}{\Delta}

\newcommand{\LL}{{\tt L}  }

\newcommand{\nn}{ {\nabla}  }

\newcommand{\pp}{ {\partial} }

\newcommand{\R} {\mathbb R}
\newcommand{\cuad}{{\sqcap\kern-.68em\sqcup}}

\newcommand{\DD}{{\mathcal D}}

\newcommand{\dist}{{\rm dist}\, }
\newcommand{\foral}{\quad\mbox{for all}\quad}
\newcommand{\ve}{\varepsilon}

\newcommand{\be}{\begin{equation}}
\newcommand{\ee}{\end{equation}}

\newcommand{\la}{\lambda}

\newcommand{\equ}[1]{(\ref{#1})}

\renewcommand{\theequation}{\thesection.\arabic{equation}}
 
 \newtheorem{lemma}{Lemma}[section]

\newtheorem{theorem}{Theorem}

\newtheorem{remark}{Remark}[section]
\newcommand{\bremark}{\begin{remark} \em}
\newcommand{\eremark}{\end{remark} }

\newtheorem{proposition}  {Proposition}[section]

\newcommand{\ov}{\overline}

\renewcommand{\O }{\Omega }

\newcommand{\N}{\mathbb{N}}

\newcommand{\e}{\varepsilon}
\newcommand{\del}{\partial}
\newcommand{\pa}{\partial}
\newcommand{\G}{\Gamma}

\newcommand{\s }{\sigma }
\renewcommand{\a }{\alpha }
\renewcommand{\b }{\beta }
\renewcommand{\d}{\delta }
\newcommand{\zn }{{x_{N}}}

\newcommand{\g }{\gamma}
\newcommand{\U}{\Upsilon}
\renewcommand{\theequation}{\thesection.\arabic{equation}}

\begin{document}

\title[Concentration at submanifolds for an elliptic Dirichlet problem]{Concentration at submanifolds for an elliptic Dirichlet problem near high critical exponents}
\author{Shengbing Deng}
\address{\noindent S. Deng -
School of Mathematics and Statistics, Southwest University,
Chongqing 400715, People's Republic of China.}
\email{shbdeng@swu.edu.cn}

\author{Fethi Mahmoudi }
\address{\noindent F.  Mahmoudi - Departamento de Ingenier\'{\i}a
Matem\'atica and
CMM, Universidad de Chile, Casilla 170 Correo 3, Santiago,
Chile.}\email{fmahmoudi@dim.uchile.cl}

\author{Monica Musso}
\address{\noindent M. Musso - {\it Corresponding author}
Departamento de Matem\'atica, Pontificia Universidad Catolica de
Chile, Avda. Vicu\~na Mackenna 4860, Macul, Chile} \email {mmusso@mat.puc.cl}



\begin{abstract}
Let $\Omega$ be a open bounded domain in $\R^n $ with smooth boundary $\pp\Omega$. We consider the equation $ \Delta u +
u^{\frac{n-k+2}{n-k-2}-\varepsilon} =0\,\hbox{ in }\,\O $, under
zero Dirichlet boundary condition, where  $\varepsilon$ is a small positive  parameter.  We assume that there is  a
$k$-dimensional closed, embedded minimal submanifold $K$ of  $\partial\O$, which is
non-degenerate, and along which a certain weighted average of sectional curvatures of $\pp\Omega$ is negative. Under these assumptions, we prove existence of a sequence $\e=\e_j$ and a solution $u_{\ve}$ which concentrate along $K$, as $\e \to 0^+$, in the sense that
$$
 |\nabla u_{\ve} |^2\,\rightharpoonup \, S_{n-k}^{\frac{n-k}{2}} \,\delta_K   \quad
\mbox{as} \ \ \ve \to 0
$$
where $\delta_K $
stands for the Dirac measure supported on  $K$ and
$S_{n-k}$ is an explicit positive constant. This result generalizes the one obtained in \cite{dmpa}, where the case $k=1$ is considered.

\end{abstract}
\maketitle


\section{Introduction and statement of  main results}
Consider  the following  nonlinear problem kwown as the Lane-Emden-Fowler problem (\cite{fowler})
\begin{eqnarray} \left\{
\begin{array}{rlllllll} \Delta u  + u^p  & = & 0,
\quad
 & u  >  0 \quad  \mbox{ in } \quad \Omega,
 \\[1mm]
                         u & = & 0\quad  &\mbox{ on }\quad
                         \partial\Omega ,
\end{array} \right.
\label{ef1}
\end{eqnarray}
where $\Omega$ is a bounded domain with smooth boundary in $\R^n$
and $p>1$.  When the exponent $p$ is subcritical ($1< p< \frac{n+2}{n-2}$),
compactness of Sobolev's embedding yields a solution as a minimizer
of the variational problem
\begin{equation} S(p) =
\inf_{u\in H_0^1(\Omega) \setminus \{0\}} \frac{\int_\Omega |\nabla
u|^2 }{\left (\int_\Omega |u|^{p+1} \right )^{2\over p+1}}.
\label{Sp}\end{equation}

For $p\ge  \frac{n+2}{n-2}$ this approach
fails and essential obstructions to existence arise: Pohozaev
\cite{Po} found that no solution to \equ{ef1} exists if the domain
is star-shaped. In contrast,  Kazdan and Warner \cite{kw} observed
that if $\Omega$ is a symmetric annulus then compactness holds for
any $p>1$ within the class of radial functions, and a solution can
again always be found by the above minimizing procedure. Compactness
in the minimization is also restored, without symmetries, by the
addition of suitable linear perturbations exactly at the critical
exponent $p=\frac{n+2}{n-2}$, as established  by Brezis and
Nirenberg \cite{bn}.

\medskip

If $p\geq {n+2 \over n-2}$, the topology and geometry of the domain play a crucial role  for the solvability of the above problem; indeed, for $p=\frac{n+2}{n-2}$,  Bahri and Coron in \cite{bc} proved the existence of solution to \equ{ef1} when the
topology of $\Omega$ is non-trivial in suitable sense.  For powers
larger than critical direct use of variational arguments seems
hopeless, and one needs more general arguments to get solvability. The presence of nontrivial topology turns out to be not sufficient
to get solvability in the supercritical situation $p> \frac{n+2}{n-2}$. In fact, for $n\ge 4$ Passaseo \cite{passaseo}
exhibits a domain constituted by a thin tubular neighborhood of a copy of the sphere $\mathbb{S}^{n-2}$ embedded in $\R^n$ for which a
Pohozaev-type identity yields that no solution exists if $p\ge
\frac{n+1}{n-3}$ ({\em the so-called second critical exponent}).
\medskip

In this paper we consider the case when $p$  is below but sufficiently close to the
$k$-th critical exponent, defined as ${n-k+2 \over n-k-2}$, with $0\le k\le  n-1$.  Namely we consider the following problem
\begin{equation}\label{eq:pe}
  \begin{cases}
    \Delta u  +u^{\frac{n-k+2}{n-k-2}-\ve} =0,\ \ \ u>0 & \text{ in } \O, \\
    u = 0 & \text{ on } \partial \O,
  \end{cases}
\end{equation}
where $\ve>0$ is a small parameter.  Assuming  that $\partial \Omega$ contains a closed minimal
non-degenerate submanifold $K$ of dimension $k$ along which a certain weighted average of sectional curvatures of $\pp\Omega$ is negative, we find a solution to \equ{ef1} which concentrates as  $p$ approaches $\frac{n+2-k}{n-2-k}$ (as $\e$ tends to $0^+$)
in a sense to be determined later. Before we state our main result, let us recall some previous works in the cases $k=0$ (point bubbling) and $k=1$ (line bubbling).

\medskip

The case $k=0$ has been extensively considered in the literature, see for instance  \cite{bp,han,rey,wei} and some references therein. It has been proven the existence of  {\em bubbling solutions}
around special points of the domain, which resemble a sharp extremal
of the best Sobolev constant in $\R^n$
$$ S_n:=\inf_{u\in {\mathcal D}^{1,2}(\R^n) \setminus \{0\}}
\frac{\int_{\R^n} |\nabla u|^2 }{\left (\int_{\R^n} |u|^{2n\over
n-2} \right )^{n-2\over n}}.
$$
The  behavior of a solution $u_\ve$  which minimizes $S(p)$ in \equ{Sp} for  $p= p_\ve =
\frac{n+2}{n-2} -\ve,$ is given by
$$
u_\ve(x) \, = \,\mu_\ve^{-{n-2 \over 2} } w_n(\mu_\ve^{-1}
(x-x_\ve)) \,+ \,o(1),\qquad \mu_\ve \, \sim \, \ve^{\frac 1{n-2}}\,
,
$$
as  $\ve \to 0^+$ , where  $w_n$ is the {\em standard bubble},
\begin{equation}
w_n(x) =\left ( \frac {c_n} {1+ |x|^2} \right )^{\frac {n-2}2 },
\qquad c_n = (n(n-2))^{1\over n-2}, \label{wn}
\end{equation}
a radial solution of
\begin{equation}\label{w0}
\Delta w + w^{\frac{n+2}{n-2}}  = 0 \quad\hbox{in } \R^n
\end{equation}
corresponding to an extremal for $S_n$, \cite{aubin,talenti}.
 The blow-up point $x_\ve$
approaches (up to a subsequences) a harmonic center $x_0$ of
$\Omega$, namely a minimizer for Robin's function of the domain, the
diagonal of the regular part of Green's function. The solution
concentrates as a Dirac mass at $x_0$, namely
\begin{equation}
|\nabla u_\ve|^2 \, \rightharpoonup\,  S_n^{\frac n2 }\,
\delta_{x_0} \quad\hbox{as } \ve\to 0\label{bub}\end{equation} in
the sense of measures. 
We also refer to \cite{blr,dfm} and to the
survey \cite{dm} for related results on construction of
point-bubbling solutions for problems near the critical exponent.


\medskip

The case $k=1$ has been  studied by del Pino-Musso-Pacard \cite{dmpa}. They proved that given a closed non-degenerate geodesic
$\Gamma$ on $\partial \Omega$, which has globally {\em negative
curvature} and assuming that a  non-resonance condition holds, then
for $n\ge 8$, problem $\equ{eq:pe}$ with $k=1$ has a
solution $u_\ve$ that satisfies
$$
|\nabla u_\ve|^2 \, \rightharpoonup \, S_{n-1}^{\frac {n-1}2}
\,\delta_\Gamma
$$
as $\ve\to 0$  in the sense of measures, where $\delta_\Gamma$ is
the Dirac measure supported on the curve $\Gamma$.

This result shows that  line-bubbling phenomenon  is conceptually quite
different to point bubbling. In fact, point concentration is determined by global information on the
domain encoded in Green's function, while only local structure of
the domain near the curve $\Gamma$ is relevant to the line-bubbling.
This is a typical phenomenon for concentration on positive dimensional sets.
Other construction of concentration along high dimensional sets under strong symmetry assumptions on the domain $\Omega$ is contained in \cite{cfp}.

\medskip
The purpose of this paper is to study existence of positive solutions to Problem \eqref{eq:pe} when $\Omega$ is a non symmetric domain in the general case $1\le k\le n-1$. Before we state our  result we need to introduce the following notations:

\medskip
Let $q \in K$. We denote by $T_q\pp\Omega$ the tangent space to
$\pp\Omega$ at the point $q$. We consider the {\em shape operator} $
\LL :  T_q\pp\Omega \to T_q\pp\Omega$ defined as
$$
\LL[e] := -\nn_e \nu(q)
$$
where $\nn_e \nu(q) $ is the directional derivative of the vector
field $\nu$ in the direction $e$. Let us consider the orthogonal
decomposition
 $$T_q\pp\Omega = T_qK\oplus N_qK $$
where $N_qK$ stands for the normal bundle of $K$. We choose
orthonormal bases
 $(e_a)_{a=1,\ldots, k}$ of $T_qK $ and  $(e_i)_{i=k+1,\ldots, n-1}$ of $N_qK$.

Let us consider the $(n-1)\times (n-1)$ matrix ${H}(q)$
representative of $\LL$ in these bases, namely
$$
H_{\a\b} (q)=  e_\a \cdot \LL[e_\b] .
$$
This matrix also represents the second fundamental form of
$\pp\Omega$ at $q$ in this basis. $H_{\a\a}(q)$ corresponds to the
curvature of $\pp\Omega$ in the direction $e_\a$. By definition, the
mean curvature of $\pp\Omega$ at $q$ is given by the trace of this
matrix, namely

$$
H_{\pp\Omega}(q) =  \sum_{\alpha=1}^{n-1} H_{\alpha\alpha}(q).
$$

In order to state our result we need to consider the mean of the
curvatures in  the directions of   $T_qK$ and $N_qK$, namely the
numbers  $\sum_{a=1}^k H_{aa}(q)$ and $\sum_{j=k+1}^{n-1}
H_{jj}(q)$.

\begin{theorem} \label{teo1}
Let $\O$ be a smooth bounded domain in $\mathbb{R}^n$, let $K$ be  $k$-dimensional non degenerate minimal submanifold  of
$\pp\Omega$.
Assume that $n-k\ge 7$ and that  the mean of the curvatures in  the directions of   $T_qK$ is negative, namely,
$$
\sum_{a=1}^k H_{aa}(q)<0\quad \mbox{for\ all}\ \ q\in K.
$$
\noindent Then, for a sequence
$\ve= \ve_j\longrightarrow \,0$, Problem \equ{eq:pe} has a positive solution
$u_{\ve}$ concentrating along $K$ as $\ve\to 0$, in the sense that
$$
 |\nabla u_{\ve} |^2\,\rightharpoonup \, S_{n-k}^{\frac{n-k}{2}} \,\delta_K   \quad
\mbox{as} \ \ \ve \to 0
$$
where $\delta_K $
stands for the Dirac measure supported on  $K$ and
$S_{n-k}$ is an explicit positive constant.
\end{theorem}

\medskip
The condition $n-k\ge 7$ appears also in many previous works like \cite{dmpa}, it is a technical condition that seems essential
for our  method to work but we believe the phenomenon described
should also be true for lower co-dimensions.  We also point out that the resonance phenomenon has already
been found  in the analysis of higher dimensional
concentration in other elliptic boundary value problems, in
particular for  Neumann singular perturbation problem in
\cite{mmah,mm1,mm2,m} and nonlinear Schr\"odinger equations on compact Riemannian manifolds without boundary or in $\R^N$, see  \cite{dkw}, \cite{masayao}.

\medskip
The solution predicted in Theorem \ref{teo1} can be described as follows:
points $x\in \R^n$ near $K     $,  can be
described as
$$
x= q + z  , \quad {\mbox {for}} \quad q \in K     , \quad |z| = \dist (x,K     ).
$$
At main order our solution will look like
\begin{equation}
\label{vancu}
u_\ve(x) \, \sim  \,\mu_\ve^{-{n-2 \over 2} } (q)  w_{n-k } \left( {
x-d_\ve  (q)  \over \mu_\ve (q)}   \right) ,
\end{equation}
as  $\ve \to 0^+$ , where  $w_{n-k} $ is the {\it standard bubble} in dimension $n-k$,
\begin{equation}
w_{n-k} (x) =\left ( \frac {c_{n-k} } {1+ |x|^2} \right )^{\frac {n-k-2}2 },
\qquad c_{n-k} = ((n-k ) (n-k-2))^{1\over n-k-2}, \label{wnk}
\end{equation}
a radial solution of the corresponding limit prblem in $\R^{n-k}$
\begin{equation}\label{limprok}
\Delta w + w^{\frac{n-k+2}{n-k-2}}  = 0 \quad\hbox{in } \R^{n-k} .
\end{equation}
In \eqref{vancu}, $\mu_\e (q)$ is a strictly positive scalar function that takes into account the invariance of \eqref{limprok} under scaling, while $d_\e (q)$ is a vector function, with values in $\R^{n-k}$, that describes the deviation of the center of the bubble in \eqref{vancu}  from the manifold $K$.

\medskip
The first main ingredient in proving our main theorem is the construction
of a very accurate approximate solution in powers of $\e$ and  $\rho=\ve^{\frac{N-1}{N-2}}$, in a neighborhood of the scaled submanifold $K_\rho=\rho^{-1}K$.
It is worth-mentioning that concentration at higher dimensional sets for some related problem with Neumann boundary conditions or on manifolds has been extensively studied in the last decade, see \cite{dpv,demamu,dmm,mmah,masayao} and some references therein. In most of the above  mentioned problems the profile has an exponential decay which is crucial in the construction of very accurate approximate solutions via an iterative scheme of Picard's type. Here instead the profile \equ{wnk} has a polynomial decay and henceforth much more refined estimates are needed to perform again an iterative procedure to improve the approximation.
Another issue is that the profile $U:= w_{n-k}$ copied and translated along $K$, as described in \eqref{vancu},  does not satisfy zero Dirichlet boundary conditions. Hence one needs  to introduce a function $\bar U$, see \eqref{ubar} for its definition, to adjust the boundary conditions, and take $U-\bar U$ to be the first approximation.  A third observation here is that since the limit problem is critical for dimension $n-k$, the linearized operator have a nontrivial kernel due to  invariances of the equation under translations and dilations. This amounts  to define some parameter functions $\mu_\e$ and some smooth normal sections $d_\e$ to guarantee the solvability of some projected problems. The condition  $
\sum_{a=1}^k H_{aa}(q)<0\quad \mbox{for\ all}\ \ q\in K,
$
that appears in the main Theorem \ref{teo1} is in fact imposed to guarantee the positivity of the main term of the dilation parameter $\mu_\e$. A more subtle issue we have to take care is the fact that the $(n-k)$-dimensional profile  is an
unstable solution to \eqref{limprok}. Indeed $w_{n-k}$
 is a Mountain-pass type solution (of Morse index one).  The linearized operator about this profile has one negative eigenvalue and as the concentration parameter $\e$
becomes smaller and smaller, this negative eigenvalue generates more and more unstable directions. This is the origin of a resonance phenomena and the reason why our result is valid only for a sequence $\e = \e_j \to 0^+$. The Morse index of our solutions diverges as $\e \to 0$.

\medskip
Once a very accurate approximate solution is constructed we can
built the desired solution by linearizing the main equation
around this approximation. The associated linear operator turns out
to be invertible with inverse controlled in a suitable norm by
certain large negative power of $\e$, provided that $\e$ remains
away from certain critical values where resonance occurs. The
interplay of the size of the error and that of the inverse of the
linearization then makes it possible a fixed point scheme.


\bigskip

The paper is organized as follows: We first introduce some
notations and conventions and we  expand the
coefficients of the metric near $K$ using geodesic normal coordinates (Fermi coordinates). We then expand  the Laplace-Beltrami operator.  Section \ref{s:aprsol} will be mainly
devoted to the construction of a local approximate solution. To perform this construction we need a solvability  theory and a-priori estimates for a certain linear operator,  which is developed in Section \ref{linearas}. In Section \ref{exct} we define globally the approximation and we write the solution to our problem as the sum of the global approximation plus a remaining term. Thus we express our original problem as a non linear problem in the remaining term and we prove our main Theorem.
To solve such problem, we need to understand the invertibility properties of another linear operator.   The Appendix in Section \ref{secapp} is devoted to prove some technical facts.

\

For brevity, most of the arguments that has been already  used in some previous works will be omitted here, referring the reader to precise references.

\

\setcounter{equation}{0}
\section{Setting up the problem in geodesic normal  coordinates}\label{sec4}

In this section we first introduce Fermi coordinates near a
$k$-dimensional submanifold of $\pa \O \subset \R^{n}$ (with
$n=N+k$) and we expand the coefficients of the metric in these
coordinates. We will omit details here referring to   \cite{demamu,dmm,mmah}. We then express our main equation in these Fermi coordinates.

\medskip
\subsection{ Notation and conventions}
\noindent  Dealing with coordinates, Greek letters like $\a, \b,
\dots$, will denote indices varying between $1$ and $n-1$, while
capital letters like $A, B, \dots$ will vary between $1$ and $n$;
Roman letters like $a$ or $b$ will run from $1$ to $k$, while
indices like $i, j, \dots$ will run between $1$ and $N-1:= n - k -
1$. $\xi_{1}, \dots, \xi_{N-1}, \xi_{N}$ will denote coordinates
in $\R^{N}=\R^{n-k}$, and they will also be written as $\bar
\xi=(\xi_{1}, \dots, \xi_{N-1})$, $\xi=(\bar \xi,\xi_N)$.
The manifold $K$ will be parameterized with coordinates
$y = (y_1, \dots, y_k)$. Its dilation $K_\rho := \frac 1 \rho K$ will be
parameterized by coordinates $z=(z_1, \dots, z_k)$ related to the
$y$'s simply by $y = \rho z$, where $\rho=\ve^{\frac{N-1}{N-2}}$.
 Derivatives with respect to the variables $y$, $z$ or
$\xi$ will be denoted by $\pa_{y}$, $\pa_z$, $\pa_\xi$, and for
brevity sometimes we might use the symbols $\pa_{a}$, $\pa_{\ov a}$
and $\pa_i$ for $\pa_{y_a}$, $\pa_{z_a}$ and $\pa_{\xi_i}$
respectively.

\medskip
\subsection{Local coordinates expansion of the metric}\label{ss:fc}
Let $K$ be a $k$-dimensional submanifold of $(\partial\O,\ov g)$
($1\le k\le N-1$), where $\bar g$ is the induced metric on $\partial \Omega$ of the standard metric in $\R^n$. We choose
along $K$ a local orthonormal frame field $((E_a)_{a=1,\cdots
k},(E_i)_{i=1,\cdots, N-1})$ which is oriented. At points of $K$, we have the natural splitting
$T\pa \O=T K \oplus N K$ where $T K$ is the
tangent space to $K$ and $N K$ represents the normal bundle, which
are spanned respectively by $(E_a)_a$ and $(E_j)_j$.

We denote by $\nabla$ the connection induced by the metric $\ov{g}$ and by
$\nabla^N$ the corresponding normal connection on the normal bundle.
Given $q \in K$, we use some geodesic coordinates $y$ centered at
$q$. We also assume that at $q$ the normal vectors $(E_i)_i$, $i =
1, \dots, n$, are transported parallely (with respect to $\nabla^N$)
through geodesics from $q$, so in particular
\begin{equation}\label{eq:parall}
    \ov g\left(\nabla_{E_a}E_j\,,E_i\right)=0  \quad \hbox{ at } q,
    \qquad \quad i,j = 1, \dots, n, a = 1, \dots, k.
\end{equation}
In a neighborhood of $q$ in $K$, we consider normal geodesic
coordinates
\[
f(y) : = \exp^K_q (y_a\, E_a), \qquad y := (y_{1}, \ldots, y_{k}),
\]
where $\exp^K$ is the exponential map on $K$ and summation over repeated
indices is understood. This yields the coordinate vector fields
$X_a : = f_* (\del_{y_a})$. We extend the $E_i$ along each $\gamma_E(s)$ so that they are parallel
with respect to the induced connection on the normal bundle $NK$.
This yields an orthonormal frame field $X_i$ for $NK$ in a neighborhood of
$q$ in $K$ which satisfies
\[
\left. \nabla_{X_a} X_i \right|_q \in T_q K.
\]
A coordinate system in a neighborhood of $q$ in $\partial\Omega$ is now defined by
\begin{equation}\label{eqF}
F(y,\bar x) := \exp^{\partial\Omega}_{f(y)}( x_i \, X_i), \qquad
(y,\bar x) :=( y_{1}, \ldots, y_{k},x_1, \ldots, x_{N-1}),
\end{equation}
with corresponding coordinate vector fields
\[
X_i : = F_* (\del_{x_i}) \qquad \mbox{and} \qquad  X_a : = F_*
(\del_{y_a}).
\]
By our choice of coordinates, on $K$ the metric $\ov{g}$ splits in
the following way
\begin{equation}\label{eq:splitovg}
    \ov g(q) = \ov g_{ab}(q)\,d y_a\otimes d y_b+\ov
g_{ij}(q)\,dx_i\otimes dx_j, \qquad \quad q \in K.
\end{equation}
We denote by $\Gamma_a^b(\cdot)$ the 1-forms defined on the normal
bundle, $NK$, of $K$ by the formula~
\begin{equation}\label{eq:Gab}
\ov g_{bc} \Gamma_{ai}^c:=  \ov g_{bc} \Gamma_a^c(X_i)=\ov g(\nabla_{X_a}X_b,X_i) \quad \hbox{at } q=f(y).
\end{equation}
Note that
$ K \hbox{ is minimal }$ if and only if $\sum_{a=1}^k\G^a_a(E_i)
    = 0
$ for any $i = 1, \dots, N-1$.

Define $q=f(y)=F(y,0)\in K$ and let $(\tilde g_{ab}(y))$ be the induced metric on $K$.
When we consider the metric coefficients in a neighborhood of $K$,
we obtain a deviation from formula \eqref{eq:splitovg}:
\[
\begin{array}{rllll}\ov g_{ij}&=\delta_{ij}+\frac{1}{3}\,R_{istj}\,x_s\,x_t\,
+\,{\mathcal O}(|x|^3);\quad
\ov g_{aj}=
{\mathcal O}(|x|^2);\\[3mm]
\ov g_{ab}&=\tilde{g}_{ab}-\bigg\{\tilde{g}_{ac}\,\Gamma_{bi}^c+\tilde{g}_{bc}\,\Gamma_{ai}^c\bigg\}\,x_i
+\bigg[R_{sabl}+\tilde g_{cd}\Gamma_{as}^c\,\Gamma_{bl}^d \bigg]x_s x_l+\,{\mathcal O}(|x|^3).
\end{array}
\]
Here $a=1,...,k$ and any $i,j=1,...,N-1$, and $R_{\a\b\g\d}$ the components of the curvature tensor
with lowered indices, which are obtained by means of the usual ones
$R_{\b\g\d}^\s$ by~
\begin{equation}
\label{ctens} R_{\a\b\g\d}=\ov
g_{\a\s}\,R_{\b\g\d}^\s.\end{equation}
The proof of these facts can be found in  Lemma 2.1 in \cite{demamu}, see also
  \cite{chavel,mmp,mazzeopacard}.

\medskip
Next we introduce a parametrization of a neighborhood in $\O$ of $ q
\in \pa \O$ through the map $\U$ given by
\begin{equation}\label{eq:fe}
    \U(y, x) = F( y, \bar x) +
x_N \nu(y, \bar x), \qquad x = (\bar x,x_N) \in \R^{N-1} \times \R,
\end{equation}
where $F$ is the parametrization introduced in \equ{eqF} and
$\nu(y,\bar x)$ is the inner unit normal to $\partial \O$ at
$F(y, \bar x)$. We have
$$
\frac{\partial \U}{\partial y_a} = \frac{\partial F}{\partial
y_a}(y, \bar x) +  \zn \frac{\partial \nu}{\partial y_a}(y, \bar x);
\qquad \qquad \frac{\partial \U}{\partial x_i} = \frac{\partial
F}{\partial x_i}(y, \bar x) +\zn \frac{\partial \nu}{\partial
x_i}(y, \bar x).
$$
Let us define the tensor matrix ${H}$ to be given by
\begin{equation}\label{eq:dn}
    d \nu_x [v] =  -{H}(x)[v].
\end{equation}
We thus find
\begin{equation}\label{eq:dfe1}
   \frac{\partial \U}{\partial y_a} = \left[ Id -  \zn
   {H}(y, \bar x) \right] \frac{\partial
F}{\partial y_a}(y, \bar x); \quad \mbox{and}\ \  \frac{\partial \U}{\partial x_i} = \left[ Id - \zn
  {H}(y, \bar x) \right] \frac{\partial
F}{\partial x_i}(y, \bar x).
\end{equation}
Differentiating $\U$ with respect to $\zn$ we also get
$
  \frac{\partial \U}{\partial \zn} = \nu(y, \bar x)
$.

Hence, letting $g_{\a \b}$ be the coefficients of the flat metric
$g$ of $\R^{N+k}$ in the coordinates $(y, \bar x,\zn)$, with easy
computations we deduce for $\tilde{y} = (y, \bar x)$ that
$$
    g_{\a\b} (\tilde{y},\zn) = \ov{g}_{\a\b} ( \tilde{y})
  - \zn \left( H_{\a\d} \ov{g}_{\d \b} + H_{\b \d} \ov{g}_{\d \a}
  \right) ( \tilde{y}) +  x_N^2 H_{\a \d} H_{\s \b}
  \ov{g}_{\d \s} (\tilde{y});  \quad  g_{\a N} \equiv 0; \quad \qquad g_{NN} \equiv 1.
$$
In the above expressions, with $\alpha$ and $\beta$ we denote any
index of the form $a=1, \ldots , k$ or $i=1, \ldots , N-1$.

For the metric $g$ in the above coordinates $(y,\bar x, x_N)$
we have the expansions
\begin{align*}
 g_{ij} =&\delta_{ij} - 2  x_N H_{ij} +
\frac{1}{3} \,R_{istj}\,x_s\,x_t +   x_N^2 (H^2)_{ij}
\,+\,{\mathcal O}((|x|^3), \quad 1\le i,j\le N-1;\\[3mm]
 g_{aj} =&-  x_N \bigg(H_{aj} + \tilde g_{ac}H_{cj}\bigg)+{\mathcal O}(|x|^2),\quad 1\le a\le k , \, 1\le j\le  N-1;\\[3mm]
 g_{ab} =&\tilde{g}_{ab}-\bigg\{\tilde{g}_{ac}\,\Gamma_{bi}^c+\tilde{g}_{bc}\,\Gamma_{ai}^c\bigg\}\,x_i
 - x_N\,\bigg\{ H_{ac}\,\tilde g_{bc}+ H_{bc}\,\tilde g_{ac}\bigg\}\\[3mm]
 &+ \bigg[R_{sabl}+\tilde g_{cd}\Gamma_{as}^c\,
\Gamma_{dl}^b \bigg]x_s x_l+x_N^2 (H^2)_{ab}\\[3mm]
&+ x_N\,x_k\bigg[H_{ac}\big\{\tilde g_{bf}\Gamma_{ck}^f+\tilde g_{cf}\Gamma_{bk}^f\big\}+
H_{bc}\big\{\tilde g_{af}\Gamma_{ck}^f+\tilde g_{cf}\Gamma_{ak}^f\big\}
\bigg]  + {\mathcal O}(|x|^3), \  1\le a,b\le k;
\\[3mm]
 g_{a N}  \equiv& 0, \quad a=1, \ldots , k; \qquad g_{i N} \equiv 0, \quad i=1, \ldots , N-1; \qquad  g_{NN} \equiv 1.
\end{align*}
In the above expressions $H_{\a  \b} $ denotes the components of the
matrix tensor ${H}$ defined in \equ{eq:dn}, $R_{istj}$ are the
components of the curvature tensor as defined in \equ{ctens},
$\Gamma_{ai}^b$ are defined in \equ{eq:Gab} and  $\tilde g_{ab}$ is the induced metric on $K$.

\medskip
Once we have the expression of the metric, it is a matter of computation to derive the Laplace Betrami operator. We shall do that in expanded and translated variables.

Let  $(y,x)\in \R^{k+N}$ be the local
coordinates along $K$ introduced in \equ{eq:fe}.
We define $\rho=\ve^{\frac{N-1}{N-2}}$ and we let $\mu_\ve $ be a positive smooth function  defined on $K$  and $d_{1,\e},\ldots,d_{N,\ve}:  K\longrightarrow\,  \R$ be smooth functions. We next introduce new functions $\tilde{\mu}_\ve$ and  $\tilde{d}_{\ell,\ve}$ so that
\begin{eqnarray}\label{tildemu-tilded}
\tilde{\mu}_\ve=\rho\mu_\ve,\quad
\hbox{and}\quad
\tilde{d}_\ve=(\ve^2\bar{d_\ve},\tilde{d}_{N,\ve}),\quad \mbox{with}\ \ \bar{d_\ve}=( d_{1,\ve},\ldots, d_{N-1,\ve}),\quad \tilde{d}_{N,\ve}=\ve d_{N,\ve}.
\end{eqnarray}
We next introduce the following change of variables
$z={y\over \rho}\in K_\rho:=\frac1\rho\,K$ and $\xi={{x-\tilde{d}_\ve}\over{\tilde \mu_\e}} \in \R^N$ and as before we write $\xi=(\bar \xi,\xi_N)$ with
\begin{equation}\label{b0}\ov\xi=\frac{\ov x-\ve^2 \bar{d_\ve}  }{\rho\mu_{\varepsilon}}, \quad \xi_N=\frac{x_N-\ve d_{N,\ve}}{\rho\mu_{\varepsilon}}.
\end{equation}
We now define the function $v$ by
\begin{equation} \label{b}
u(z,\ov x, x_N)=(1+\alpha_\ve)\,\tilde{\mu}_{\varepsilon}^{-\frac{N-2}{2}} \, v\left(z, \,
\xi \right).
\end{equation}
 In (\ref{b}), $\alpha_\ve$ is some parameter which has to be chosen so that
$$
\Delta((1+\alpha_\ve)U)+ \rho^{\frac{N-2}{2}\ve}\left( (1+\alpha_\ve)U \right)^{p-\ve}=0\quad \hbox{in }\, \R^N
$$
where $U$ is standard bubble in $\R^N$ defined in \eqref{wn} ($U=w_N$).
This parameter  can be computed explicitly as
$$
\alpha_\ve=\rho^{\frac{(N-2)^2}{8-2\ve(N-2)}\ve}-1.
$$
Let us mention here that with the above change of variables the functions $\tilde \mu_\e$ and $\tilde{d}_\ve$ depends slowly on the variable $z$.
To emphasize the dependence of the above change of variables on
$\mu_\ve$ and $d_\e=(\bar{d_\ve},d_{N,\ve})$, we will use the notation
\begin{equation} \label{defTT}
u={\mathcal T}_{\mu_\ve , d_\e} (v)  \quad \Longleftrightarrow u
\quad {\mbox {and}} \quad v \quad {\mbox {satisfy \equ{b}}}.
\end{equation}

We assume now that the functions $\mu_\e$ and
$d_\e$ are uniformly bounded, as $\ve \to 0$, on $K$. Since the original variables
$(y,\bar{x},x_N)\in \R^{k}\times \R^{N-1}\times \R_+$ are local coordinates along $K$, we let the
variables $(z, \xi )$ vary in the set $\DD $ defined by
\begin{equation}
\label{defD} \DD  = \left\{ (z, \bar \xi , \xi_N ) \, : \, \rho z \in K,
\quad |\bar \xi | <{\delta \over \rho }, \quad -\frac{ \tilde{d}_{N,\ve}}{  \tilde{\mu}_{\varepsilon}}< \xi_N < {\delta
\over \rho } \right\}
\end{equation}
for some  fixed positive number $\delta$. We will also use
the notation $ \DD  = K_\rho \times \hat \DD $, where
$$
\hat \DD  = \left\{ (\bar \xi , \xi_N ) \, : \, |\bar \xi | <{\delta \over \rho }, \quad -\frac{ \tilde{d}_{N,\ve}}{  \tilde{\mu}_{\varepsilon}}< \xi_N < {\delta
\over \rho } \right\}.
$$

Using the expansions of the metric we can expand the Laplace Beltrami
operator in the new variables $(z, \xi )$ in terms  of parameter functions $\tilde\mu_\ve (y)$ and
$\tilde d_\ve (y)$.
This is the content of next Lemma, whose proof can be seen in \cite[Lemma 3.3]{demamu}.

\begin{lemma} \label{scaledlaplacian}
Given  the change of variables defined  in \eqref{b}, the following
expansion for the Laplace Beltrami operator holds true
\begin{equation}
\label{lap1}
 (1+\alpha_\ve)^{-1}\mu_\ve^{{N+2 \over 2}} \Delta u =  {\mathcal A}_{\mu_\ve , d_\e } (v) :=  {\mu_\ve^2}  \D_{K_\rho} v +
\Delta_{\xi} v + \sum_{\ell=0}^5 {\mathcal A}_\ell v + B(v).
\end{equation}
Above, the expression ${\mathcal A}_k$ denotes the following
differential operators
\begin{align*}
{\mathcal A}_0 v
 = &  \tilde{\mu}_\ve D_{\ov \xi}\, v \,[\D_K \tilde{d}_\e ]  -
\tilde{\mu}_\ve \,\D_K
\tilde{\mu}_\ve \,\left( \gamma
v + D_\xi v \, [\xi] \right)    \\[3mm]
 & +   \,| \nabla_K\tilde{\mu}_\ve|^2 \left[ D_{\xi\xi} v \, [
\xi]^2 +
2 (1+\gamma ) D_\xi v [\xi] + \gamma (1+ \gamma ) v \right]\\[3mm]
    & -   \nabla_K \tilde{\mu}_\ve \,\cdot\,\left\{ 2D_{\ov\xi\,\ov\xi} v[\ov\xi] + N
D_{\ov\xi} v
\right\}\, [\nabla_K  \tilde{d}_\ve ] +   D_{\ov\xi\,\ov\xi} v \,[\nabla_K \tilde{d}_\e]^2\\[3mm]
   &- 2\, \tilde{\mu}_\ve \,\tilde g^{ab}\,\left[  D_\xi (\frac{1}{\rho}\del_{\bar a} v ) [\del_b
\tilde{\mu}_\ve \xi] + D_{  \xi} (\frac{1}{\rho}\del_{\bar a} v )[ \del_b
\tilde{d}_\e ] +
 \gamma  \del_a\tilde{\mu}_\ve\, (\frac{1}{\rho}\del_{\bar b} v )\right],
\end{align*}
where we have set $\gamma=\frac{N-2}{2}$,
\smallskip
\begin{align*}
 {\mathcal A}_1 \, v   = & \,\sum\limits_{i,j} \bigg[ 2(\tilde{\mu}_\ve \xi_N+\tilde{d}_{N,\ve})
H_{ij}   -{1\over 3}\,\,\sum\limits_{m,l} R_{mijl}
(\tilde{\mu}_\ve \xi_m + \tilde{d}_{m,\e} )
(\tilde{\mu}_\ve \xi_l + \tilde{d}_{l,\e} )  \\[3mm]
   &+ (\tilde{\mu}_\ve \xi_N+\tilde{d}_{N,\ve}) \,Q(H)_{ij}+(\tilde{\mu}_\ve \xi_N+\tilde{d}_{N,\ve})\,\sum\limits_{l} \mathfrak{D}_{Nl}^{ij}\,( \tilde{\mu}_\ve \xi_l
+\tilde{d}_{l,\e}x )
  \bigg] \del^2_{ij} v,
\end{align*}
where the function $Q(H)_{ij}$ is defined as
\begin{equation*}\label{eq:Qij}
 Q(H)_{ij}=3x_N^2 \,H_{ik}\,H_{kj}+x_N^2\,\bigg( 2\,H_{ia}\,H_{aj}+\tilde g^{ab}\,H_{ia}\,H_{bj} \bigg),
\end{equation*}
and the functions $\mathfrak{D}_{Nk}^{ij}$ are smooth functions of
the variable $z= \frac{y}{\rho}$ and uniformly bounded. Furthermore,
\smallskip
\begin{equation*}
\label{D4} {\mathcal A}_2 v =\tilde{  \mu}_\ve \sum\limits_j \bigg[ \sum_s
\frac23 R_{mssj}+\sum\limits_{m,a,b}
 \big( \tilde g_\e^{ab}\, R_{mabj}- \G_{am}^b
\G_{bj}^a \big)\bigg](\tilde{\mu}_\ve\xi_m+\tilde{d}_{m,\e})
 \del_j v,
\end{equation*}
and
\begin{equation*}
\label{D5} {\mathcal A}_3 v \, =-\tilde{\mu}_\ve \bigg[  {\rm tr }(H)
+(\tilde{\mu}_\ve+\tilde{d}_{N,\ve}) {\rm tr }(H^2)\,  \bigg]
 \del_N v.
\end{equation*}
Moreover
\smallskip
\begin{eqnarray*}
\label{D2} {\mathcal A}_4 v  =
2(\tilde{\mu}_\ve+\tilde{d}_{N,\ve})(H_{aj}+\tilde{g}^{ac}H_{cj})   \left(\frac{\tilde{\mu}_\ve}{\rho} \partial^2_{\bar{a}j}v-\partial_a\tilde{\mu}_\ve D_\xi(\partial_jv)-D_\xi(\partial_jv)[\partial_a \tilde{d}_\ve]+(1+\gamma)\partial_a\tilde{\mu}_\ve\partial_j v\right)
\end{eqnarray*}
and
\smallskip
\begin{equation*}
{\mathcal A}_5 v    =  \left( \sum\limits_{a,j}
\mathfrak{D}_j^a   [\tilde{\mu}_\ve  \xi_j + \tilde{d}_{j,\e}] + (\tilde{\mu}_\ve\xi_N+\tilde{d}_{N,\ve})\,
\mathfrak{D}_N^a   \right)   \left\{ \tilde{\mu}_\ve \left[ -   D_{\ov\xi} v \, [\del_a \tilde{d}_\e ] +
\tilde{\mu}_\ve \del_{\bar a} v -   \del_a \tilde{\mu}_\ve (\gamma v +
D_\xi v \, [\xi] ) \right] \right\}
\end{equation*}
where $\mathfrak{D}_j^a$ and $\mathfrak{D}_N^a$ are smooth functions
of $z= \frac{y}{\rho}$. Finally,
the operator $B(v)$ is defined below,
\begin{align}\label{decbv}
\mathcal{B}(v) =&O \left(   ( \tilde{\mu}_\ve \bar \xi  + \bar{\tilde{d}}_\ve )^2 + (\tilde{\mu}_\ve
\xi_N+\tilde{d}_{N,\ve})(\tilde{\mu}_\ve \bar \xi
+ \bar{\tilde{d}}_\ve )+(\tilde{\mu}_\ve
\xi_N+\tilde{d}_{N,\ve})^2   \right) \times\nonumber\\
 & \times\left( -\frac{N}{2}\,\pa_{\bar a}\tilde{\mu}_\ve\,\pa_lv
+\frac{\tilde{\mu}_\ve}{\ve}\,\pa^2_{\bar al}v
-\pa_{\bar a}\tilde{\mu}_\ve\xi_J\pa^2_{lJ}v-\pa_{\bar a}\Phi^j\pa^2_{lj}v \right)\nonumber\\
 &+O \left( ( \tilde{\mu}_\ve \bar \xi  + \bar{d}_\ve )^3   + (\tilde{\mu}_\ve   \xi_N
+  \tilde{d}_{N,\ve} ) ( \tilde{\mu}_\ve \bar \xi  + \bar{\tilde{d}}_\ve )^2\right.\nonumber\\
 &  \left.+ (\tilde{\mu}_\ve   \xi_N
+  \tilde{d}_{N,\ve} )^2 ( \tilde{\mu}_\ve \bar \xi  + \bar{\tilde{d}}_\ve )  \tilde{d}_\ve |+ (\tilde{\mu}_\ve   \xi_N
+  \tilde{d}_{N,\ve} )^3  \right) \del^2_{ij} v,
\end{align}
where $\tilde{\bar{d}}_\ve =\ve^2\bar{d}_\ve$.
We recall that the symbols $\pa_{a}$, $\pa_{\ov a}$ and $\pa_i$
denote the derivatives with respect to $\pa_{y_a}$, $\pa_{z_a}$ and
$\pa_{\xi_i}$ respectively.
\end{lemma}

\subsection{Expressing the equation in coordinates}

We recall that we want to find a solution to  the problem
\begin{equation}\label{eq:pe1}
     \Delta u +u_+^{p-\ve} =0  \text{ in } \O , \quad
    u = 0  \text{ on } \partial \O,
\end{equation}
where $p=\frac{N+2}{N-2}$ with $N=n-k$.

After performing the change of variables in
\eqref{b}, the original equation in $u$  reduces locally close to
$K_\rho = {K\over \rho}$ to the following equation in $v$
\begin{equation}
\label{adesso}
 -{\mathcal A}_{\mu_\ve , d_\e } v  -  \mu_\ve^{\frac{N-2}{2}\ve}v^{p-\ve} =0 ,
\end{equation}
where ${\mathcal A}_{\mu_\ve , d_\ve }$ is defined in \eqref{lap1}
and $p={N+2 \over N-2}$. We denote by ${\mathcal S}_\ve$ the
operator given by \equ{adesso}, namely
\begin{equation}
\label{Sep} {\mathcal S}_\ve (v ) := -{\mathcal A}_{\mu_\ve ,
d_\e } v -  \mu_\ve^{\frac{N-2}{2}\ve}v^{p-\ve}.
\end{equation}
Recalling the definitions $\tilde{\mu}_\ve=\rho\mu_\ve$,  $\tilde{d}_\ve=(\tilde{\bar{d}}_\ve,\tilde{d}_{N,\ve})=(\ve^2\bar{d}_\ve,\ve d_{N,\ve})$ with $\bar{d}_\ve=(d_{1,\ve},\ldots,d_{N-1,\ve})$ in Lemma \ref{scaledlaplacian}, we get the following result which gives the expansion of ${\mathcal S}_\ve (v )$ in powers of $\e$, $\rho$ and in terms of the real function $\mu_\e$, $d_{N,\e}$ and the normal section $\bar d_\ve$.

\begin{lemma} \label{scaledlaplacianco}
It holds that
\begin{align}
\label{lap1a}
 {\mathcal S}_\ve (v ) =&-  {\mu_\ve^2}  \D_{K_\rho} v -
\Delta_{\xi} v-\mu_\ve^{\frac{N-2}{2}\ve}v^{p-\ve} -\ve\, 2 d_{N,\ve} H_{ij}\partial_{ij} v \nonumber\\
 &- \ \rho\, \{2\mu_\ve \xi_N H_{ij}\partial_{ij} v-\mu_\ve tr(H)\partial_N v\}\\
 &-\ve^2\, \mathcal{S}_1(v )-\ve\rho\, \mathcal{S}_2(v )-\rho^2\mathcal{S}_3(v)-\ve^3\mathcal{S}_4(v )-\ve^2\rho\mathcal{S}_5(v )-\ve^4\mathcal{S}_6(v )-B(v),\nonumber
\end{align}
where the terms $\mathcal{S}_j(v)$ are given by
\begin{eqnarray*}
\label{lap1as14}
 &&\mathcal{S}_1(v) =   |\nabla_K d_{N,\ve}|^2 \partial_{NN}^2v+d_{N,\ve}^2Q(H)_{ij}\partial_{ij}^2 v-2d_{N,\ve}(H_{aj}+\tilde{g}^{ac}H_{cj})[\partial_ad_{N,\ve}\partial^2_{Nj}v-\frac{1}{\ve}\mu_\ve \partial_{\bar{a}j}^2v], \nonumber\\[3mm]
&&\mathcal{S}_2(v) = -\mu_\ve \partial_N v[\triangle_K d_{N,\ve}]+2(1+\gamma)\nabla_K\mu_\ve \partial_N v[\nabla_K d_{N,\ve}]+2\nabla_K\mu_\ve \partial_{\xi\xi_N}^2v[\xi,\nabla_Kd_{N,\ve}] \nonumber\\
&&\qquad\quad -2\mu_\ve\tilde{g}^{ab}\frac{1}{\rho}\partial_{N\bar{a}}^2v\partial_bd_{N,\ve} +2\mu_\ve d_{N,\ve}\xi_NQ(H)_{ij}\partial_{ij}^2v-\mu_\ve d_{N,\ve}tr(H^2)\partial_Nv \nonumber\\
&&\qquad\quad  -2(H_{aj}+\tilde{g}^{ac}H_{cj})\Big[\mu_\ve\xi_N\partial_ad_{N,\ve}\partial_{Nj}^2v+(1+\gamma)d_{N,\ve}\partial_a\mu_\ve\partial_j v+\frac{1}{\ve}\mu_\ve^2\xi_N \partial_{\bar{a}j}^2v\Big],\nonumber\\[3mm]
&&\mathcal{S}_3(v) =\mu_\ve\triangle_K \mu_\ve[\gamma v+D_\xi v[\xi]]-2\gamma\mu_\ve\nabla_K\mu_\ve \nabla_{K_\rho}v-2\mu_\ve\tilde{g}^{ab}\partial_b\mu_\ve D_\xi(\frac{\partial_{\bar{a}}v}{\rho})[\xi]\nonumber\\
&&\quad\quad \quad+|\nabla_K\mu_\ve|^2\left[\gamma(\gamma+1)v+2(\gamma+1)D_\xi v[\xi]+D_{\xi\xi}^2v[\xi]^2\right]\nonumber\\
&&\quad\quad \quad -\frac{1}{3}R_{islj}\mu_\ve^2\xi_s\xi_l\partial_{ij}^2v+\mu_\ve^2\xi_N^2Q(H)_{ij}\partial_{ij}^2v+\mu_\ve^2\xi_N\mathfrak{D}_{Nl}^{ij} \xi_l\partial_{ij}^2v\nonumber\\
&&\quad\quad\quad  \left. +\mu_\ve^2[\frac{2}{3}R_{mllj}+\tilde{g}^{ab}R_{jabm}-\Gamma^c_{am}\Gamma^{a}_{cj}]\xi_m\partial_jv
-\mu_\ve^2\xi_Ntr(H^2)\partial_Nv\right.\nonumber\\
&&\quad\quad \quad  -2(H_{aj}+\tilde{g}^{ac}H_{cj})[\mu_\ve\xi_N\partial_a\mu_\ve D_\xi(\partial_jv)+(1+\gamma)\xi_N\mu_\ve\partial_a\mu_\ve \partial_jv],\nonumber\\[3mm]
&&\mathcal{S}_4(v) = \partial_{jN}^2v\nabla_K d_{N,\ve}\nabla_Kd_j+d_N\mathfrak{D}_{Nl}^{ij} d_{l,\ve}\partial_{ij}^2v-2d_{N,\ve} (H_{aj}+\tilde{g}^{ac}H_{cj})\partial_ad_l\partial_{jl}^2v,\nonumber\\[3mm]
&&\mathcal{S}_5(v) = -\mu_\ve\partial_jv\triangle_K d_{j,\ve}+\gamma(1+\gamma) \nabla_K\mu_\ve \nabla_Kd_{j,\ve} \partial_jv+2\nabla_K\mu_\ve\nabla_Kd_j\partial_{jl}^2v\xi_l \nonumber\\
&&\qquad \qquad\left. -2\mu_\ve\tilde{g}^{ab}\frac{1}{\rho}\partial_{\bar{a}j}^2v\partial_bd_{j,\ve}-\frac{1}{3}\mu_\ve R_{mijl}(\xi_md_{l,\ve}+\xi_ld_{m,\ve})\partial_{ij}^2v
+\mu_\ve\mathfrak{D}_{Nl}^{ij}\xi_Nd_{l,\ve}\partial_{ij}^2v\right.\nonumber\\
&&\qquad \qquad+\mu_\ve[\frac{2}{3}R_{mllj}+\tilde{g}^{ab}R_{jabm}-\Gamma^c_{am}\Gamma^{a}_{cj}]d_{m,\ve}\partial_jv
-2\mu_\ve\xi_N(H_{aj}+\tilde{g}^{ac}H_{cj})\partial_ad_{l,\ve}\partial_{jl}^2v,\nonumber\\[3mm]
&&\mathcal{S}_6(v) = \nabla_Kd_{j,\ve}\nabla_Kd_{i,\ve} \partial_{ij}^2v-\frac{1}{3}R_{islj}d_{s,\ve}d_{l,\ve}\partial_{ij}^2v,
\end{eqnarray*}
where  the functions $\mathfrak{D}_{Nk}^{ij}$ are smooth functions of
the variable $z= \frac{y}{\rho}$ and uniformly bounded. Finally,
the operator $B(v)$ is defined  in \eqref{decbv}. \smallskip

\end{lemma}

\setcounter{equation}{0}
\section{Construction of local approximate solutions}\label{s:aprsol}
\smallskip

In this section, we will construct very accurate approximate solutions to our problem.
The basic tool for this construction is a linear theory we describe
below. We consider the domain $\DD$ defined as (\ref{defD})
and for functions $\phi$ defined on $\DD$, an operator of the form
\begin{equation} \label{rivoli2}
L(\phi):=
- \Delta_\xi \phi - p U^{p-1} \phi   + b_{ij}(\rho z,\xi) \partial_{ij}\phi +b_i(\rho z,\xi)\partial_i\phi
\end{equation}
where  $b_{ij} $, $b_i$ and $c$ are functions defined in
$\DD$, which depend smoothly on $y \in K$. Recall that
a variable $z\in K_\rho$ has the form $\rho z=y \in K$.

\medskip

We want to establish a solvability theory and an a-priori bounds for the  following linear problem
\begin{equation}\label{eq:eqwd}
 \left\{
    \begin{array}{ll}
    L(\phi)=h,  &\  \hbox{ in } \DD\\[3mm]
     \phi  = 0, \ \, &\ \hbox{ on } \partial\hat\DD\\[3mm]
    \int_{\hat\DD} \phi (\rho z, \xi ) Z_j (\xi ) \, d\xi = 0 &\  \forall z \in K_\rho, \quad  j=0, \ldots N+1,
    \end{array}
  \right.
\end{equation}
for a given function $h : \DD \to \R$, which depends smoothly on the variable $y \in K$.
The functions $Z_j (\xi )$, $j=1, \ldots , N+1$, are
 \begin{equation}
 \label{lezetas}
Z_j (\xi ) = {\partial U \over \partial \xi_j} , \quad j=1, \ldots
, N, \quad Z_{N+1} (\xi ) = \xi \cdot \nabla U
(\xi )  + \frac {N-2}2 U (\xi ) .
\end{equation}
It is well known  (see for instance \cite{bianchiengel}) that these functions
are the only bounded solutions to the linearized equation around
$U$ of problem \equ{w0} in $\R^N$
$$
-\Delta \phi - p U^{p-1} \phi =0 \quad {\mbox {in}} \quad \R^{N}.
$$
Moreover, $Z_0$ is the first eigenfunction  (normalized to have $L^2$-norm equal
to $1$) corresponding to the first eigenvalue
$\lambda_1 >0$
$L^2(\R^N)$  of the problem
\begin{equation}\label{denotareZ0}
 \Delta_\xi \phi + pU(\xi)^{p-1}\phi  = \la \phi\quad   \mbox{ in} \quad
 \R^{N}.
\end{equation}
Observe that this eigenfunction
decays exponentially at infinity like $\xi \longmapsto |\xi|^{-{N-1 \over 2}}
e^{-\sqrt{\lambda_1} \, |\xi|}$.

\medskip
\noindent
In order to solve the above linear problem, we define the following norms.
Let $\delta >0$ be a positive, small fixed number. Let $r$ be an
integer. For a function $w$ defined in $\DD=K_\rho \times \hat\DD$, we define
\begin{equation}\label{eqinftynu}
\|w\|_{\e,r}:=\sup_{(z,\xi)\in K_\rho \times \hat\DD}\left( \,(1+|\xi|^2)^{r \over 2}|w(z,\xi)|
\right).
\end{equation}
Let $\sigma \in (0,1)$. We define
\begin{equation}
\label{normsigma} \| w \|_{\e, r, \sigma} := \| w \|_{\rho , r} + \sup_{(z,\xi)\in K_\rho
\times \hat\DD}\left( \,(1+|\xi|^2)^{r
+\sigma\over 2} [w]_{\sigma, B(\xi , 1)} \right)
\end{equation}
where we have denoted
\begin{equation}\label{eqnorms}
[w]_{\sigma, B(\xi , 1)} := \sup_{ \xi_1, \xi_2\in B(\xi , 1),\, \xi_1\ne \xi_2}
\frac{|w(z,\xi_2)-w(z,\xi_1)|}{|\xi_1-\xi_2|^\sigma}.
\end{equation}
We will establish  existence and uniform a priori estimates for problem (\ref{eq:eqwd}) in the above norms, provided that appropriate bounds for coefficients hold.
We have the validity of the following result.

\medskip

\begin{proposition}\label{linear}
Assume that $N\geq 7$, and let $r$ be an integer such that $2<r < N-2$.
Then there exist positive numbers $\delta,  C$ such that if, for all
$i,j$
\begin{equation}\label{ipo1}
\begin{array}{lll}
 \|b_{ij}\|_\infty + \|D b_{ij}\|_\infty +
\|(1+|y|)b_i\|_\infty 
< \delta,
\end{array}
\end{equation}
for all $y=\rho z\in \R^k$.
Let $h : K \times \hat\DD
\to \R $ be a function that depends smoothly on the variable $y \in
K$, such that $ \| h \|_{\ve , r} $ is bounded, uniformly in $\ve$,
and
$$
\int_{\hat\DD}h(\ve z,\xi)Z_j(\xi)d\xi=0\quad \mbox{for\ all}\ z\in K_\rho,\ \ j=0,1,\ldots,N+1.
$$
Then there exists a solution $\phi$ of problem (\ref{eq:eqwd}) and a constant $C>0$ such that
\begin{equation}
\label{est0a} \| D^2_\xi \phi \|_{\ve , r , \sigma } + \| D_\xi \phi
\|_{\ve , r -1 , \sigma } +\|\phi \|_{\ve, r- 2 , \sigma}\le C
\|h\|_{\ve,r ,\sigma} .
\end{equation}
Furthermore, the function $\phi$ depends smoothly on the variable
$\rho z$, and the following estimates hold true: for any integer $l$
there exists a positive constant $C_l$ such that
\begin{equation}
\label{est1a} \| D^l_y \phi \|_{\ve , r- 2 , \sigma }  \le C_l \left(
\sum_{k\leq l} \|D^k_y h\|_{\ve,r , \sigma}\right).
\end{equation}
\end{proposition}

\begin{proof}
The proof is adapted from Proposition 3.1 in \cite{dmm}. We give here the main ideas of the proof for completeness.

First, we prove a priori estimates by dividing into the following several steps.

{\it Step 1}. \ \ Let us assume that in problem (\ref{eq:eqwd}) the coefficients $b_{ij},b_i$ are identically zero.
Thus assume that $\phi$ is a solution to
\begin{equation}\label{eq:eqw}
 \left\{
    \begin{array}{lll}
   - \D \phi - p w_0^{p-1} \phi 
   =h  &\quad   \hbox{ in } \DD\\
     \phi  = 0, \ \, &\quad\hbox{ on } \partial\hat\DD\\
    \int_{\hat\DD} \phi (\rho z, \xi ) Z_j (\xi ) \, d\xi = 0 & \foral z \in K_\rho, \quad  j=0, \ldots N+1.
    \end{array}
  \right.
\end{equation}
We  claim that there exists $C>0$ such that
\begin{equation}
\label{mar2} \|\phi \|_{\ve, r- 2}\le C \|h\|_{\ve,r}.
\end{equation}
By contradiction, assume that there exist sequences $\ve_n
\to 0$(note that $\rho_n=\ve_n^{\frac{N-1}{N-2}}\to 0$), $h_n$ with $\| h_n \|_{\ve_n , r} \to 0$ and solutions
$\phi_n$  to \equ{eq:eqw} with $\| \phi_n \|_{\ve_n ,  r -2} =1$.

Let $z_n \in K_{\rho_n}$ and $\xi_n$ be such that
$
|\phi_n (\rho_n z_n , \xi_n )| = \sup  |\phi_n (y, \xi )|.
$
We may assume that, up to subsequences, $ \rho_n z_n  \to \bar y $ in
$K$. In particular one
gets that $|\xi_n | \leq C \rho_n^{-1}$ for some positive
constant $C$ independent of $\e_n$.

Let us now assume that there exists a positive constant $M$ such
that $|\xi_n |\leq M$. In this case, up to subsequences, one gets
that $\xi_n \to \xi_0$. Consider the functions
$\tilde \phi_n ( z, \xi ) = \phi_n ( z , \xi + \xi_n ).$ This is a
sequence of uniformly bounded functions, that converges uniformly over compact sets of $K \times
\hat\DD$ to a function $\tilde \phi$ solution to
$
    - \D \tilde \phi - p w_0^{p-1} \tilde \phi =0 $   in $ \R^{N}.
$
Since the orthogonality conditions pass to the limit, we get that
furthermore
$$
    \int_{\R^{N} } \tilde \phi (y, \xi ) Z_j (\xi ) \, d\xi = 0 \quad \foral y \in K, \quad \foral j=0, \ldots N+1.
    $$
These facts imply that $\tilde \phi \equiv 0$, that is a
contradiction.

Assume now that $\lim\limits_{n \to \infty} |\xi_n | = \infty$. Consider the scaled function
$$
\tilde \phi_n (z, \xi ) = \phi_n (z, |\xi_n | \xi + \xi_n )
$$
defined on the set
$$
\bar\DD=\left\{(z,\bar\xi,\xi_N) :\ |\bar\xi|<\frac{\delta}{\rho_n|\xi_n|}-\frac{\xi_n}{|\xi_n|},\
-\frac{\ve_nd_{N,\ve_n}}{\rho_n\mu_{\ve_n}|\xi_n|}-\frac{\xi_n}{|\xi_n|}<\xi_N<\frac{\delta}{\rho_n|\xi_n|}-\frac{\xi_n}{|\xi_n|}\right\}.
$$
Thus
$\tilde \phi_n$ satisfies the equation
$$
-\Delta \tilde \phi_n - p \,c_N^{p-1} {|\xi_n |^2 \over (1+| \, |\xi_n | \xi
+ \xi_n |^2 )^2 } \tilde \phi_n  =
|\xi_n |^2 h (z, |\xi_n | \xi +\xi_n )\ \ \mbox{in}\ \bar\DD.
$$
Under our assumptions, we have that $\tilde \phi_n$ is
uniformly bounded and it converges locally over compact sets to
$\tilde \phi$ solution to
$
\Delta \tilde \phi = 0, \quad  |\tilde
\phi | \leq C |\xi |^{2-r}\quad {\mbox {in}} \quad \R^{N}.
$
Since $2<r<N$, we conclude that $\tilde \phi \equiv 0 $, which is a
contradiction.
The proof of \equ{mar2} is completed.

\smallskip
{\it Step 2}. \ \  We shall now show that there exists $C>0$ such
that, if $\phi$ is a solution to \equ{eq:eqw}, then
\begin{equation}
\label{mar3} \| D^2_\xi \phi \|_{\ve , r } + \| D_\xi \phi \|_{\ve ,
r -1  } +\|\phi \|_{\ve, r- 2 }\le C \|h\|_{\ve,r }.
\end{equation}

For $z \in K_\rho$, we have that $\phi$ solves $ -\Delta_{\xi} \phi = h+p w_0^{p-1} \phi:=\tilde h $
in $ \DD $. From Step 1, we have that $|\tilde h |\leq \frac{\|h\|_{\ve,r}}{(1+|\xi|^{r } )}$. Elliptic estimates give that $| \phi | \leq
{C \over (1+|\xi|^{r -2 } )}$.

Let us now fix a point $e \in \R^N$ and a positive number $R>0$.
Perform the change of variables $ \tilde \phi (z,t) =R^{r-2}\, \phi (z, Rt
+3Re)$, so that
$$
\Delta \tilde \phi =   R^{r}  \tilde h(z,Rt+3Re) \quad {\mbox {in}}
\quad |t|\leq 1.
$$
Elliptic estimates
give then that
$$\| D^2 \tilde \phi \|_{L^\infty(B(0,1))} +\| D \tilde \phi \|_{L^\infty(B(0,1))} \leq C
\|R^{r}  \tilde h(z,Rt+3Re) \|_{L^\infty (B(0,2))}.$$
It then follows that
$$
\| (1+|\xi |)^{r} D^2 \phi \|_{L^\infty (|\xi |\leq \delta
\rho^{-1} )} \leq C \| (1+|\xi |)^{r} h \|_{L^\infty (|\xi
|\leq \delta \rho^{-1} )}.
$$

Arguing in a similar way, one gets the internal weighted estimate
for the first derivative of $\phi$
$$
\| (1+|\xi |)^{r-1} D \phi \|_{L^\infty (|\xi |\leq \delta
\rho^{-1} )} \leq C \| (1+|\xi |)^{r} h \|_{L^\infty (|\xi
|\leq \delta \rho^{-1} )}.
$$

By using the
representation formula for solution $\phi$ to the above equation, we
see that $ | \phi | \leq C \rho^{{r - 2 \over 2}}$ in $|\xi |<\delta
\rho^{-1}$. Furthermore, elliptic estimates give that in this
region $ | D\phi | \leq C \rho^{{r - 1 \over 2}}$ and $ | D^2 \phi |
\leq C \rho^{{r  \over 2}}$. This concludes the proof of \equ{mar3}.

\smallskip

{\it Step 3}. \ \  We shall now show that there exists $C>0$ such
that, if $\phi$ is a solution to \equ{eq:eqw}, then
\begin{equation}
\label{mar4} \| D^2_\xi \phi \|_{\ve , r , \sigma } + \| D_\xi \phi
\|_{\ve , r -1 , \sigma } +\|\phi \|_{\ve, r- 2 , \sigma}\le C
\|h\|_{\ve,r ,\sigma} .
\end{equation}

Let us first assume we are in the region $|\xi |<\delta
\rho^{-1}$, and $z \in K_\rho$. We first claim that from
elliptic regularity, we have that if  $ \|h\|_{\ve,r ,\sigma} \leq
C$ then $\|\phi \|_{\ve, r- 2 , \sigma} \leq C$. Thus, we write that
$\phi$ solves $ -\Delta \phi = \tilde h $ in $ |\xi |< \delta
\rho^{-1} $ where $\|\tilde h\|_{\ve,r ,\sigma} \leq C$.

Arguing as in the previous step, we fix a point $e \in \R^N$ and a
positive number $R>0$. Perform the change of variables $ \tilde \phi (z,t) =R^{r-2}\, \phi (z, Rt
+3Re)$, so that
$$
\Delta \tilde \phi = R^{r}  \tilde h(z,Rt+3Re) \quad {\mbox {in}}
\quad |t|\leq 1.
$$
Elliptic estimates
give then that $\|  D^2 \tilde \phi \|_{C^{0,\sigma}
(B(0,1))} \leq C \| \tilde h \|_{C^{0,\sigma} (B(0,2))}$. This implies
that
$$
  \| D_\xi^2 \tilde \phi \|_{L^\infty (B_1)} + [D^2
\tilde \phi ]_{\sigma, B(0,1)} \leq C.
$$
In particular, we have for any $z \in K_\rho$, that
$$
  \sup_{y_1 , y_2 \in B(0,1)} {|D^2 \tilde \phi (z, y_1) - D^2
\tilde \phi (z, y_2) | \over |y_1 - y_2 |^\sigma} \leq C.
$$
This inequality gets translated in term of $\phi$ as
$$
R^{r+\sigma} \sup_{\xi_1 , \xi_2 \in B(\xi,1)} {|D^2  \phi (z,
\xi_1) - D^2  \phi (z, \xi_2) | \over |\xi_1 - \xi_2 |^\sigma} \leq
C.
$$
In a very similar way, one gets the estimate on $D\phi$.
This concludes the
proof of \equ{mar4}.

\smallskip

{\it Step 4}. \ \ Differentiating equation \equ{eq:eqw} with respect to the
$z$ variable $l$ times and using elliptic regularity estimates, one
proves that
\begin{equation}
\label{est1} \| D^l_y \phi \|_{\ve , r- 2 , \sigma }  \le C_l \left(
\sum_{k\leq l} \|D^k_y h\|_{\ve,r , \sigma}\right)
\end{equation}
for any given integer $l$.

\smallskip

{\it Step 5}. \ \ Assume now that the function $b_{ij}$ and $b_i$ in \equ{eq:eqwd} are not zero, and assume that $\phi$ is a solution of problem (\ref{eq:eqwd}), then by \equ{mar4} we obtain
\begin{align*}
 & \| D^2_\xi \phi \|_{\ve , r , \sigma } + \| D_\xi \phi
\|_{\ve , r -1 , \sigma } +\|\phi \|_{\ve, r- 2 , \sigma}\nonumber\\[3mm]
 \le&  C
\|h\|_{\ve,r ,\sigma}+C\|b_{ij}\partial_{ij}\phi\|_{\ve,r ,\sigma}+C\|b_{i}\partial_{i}\phi\|_{\ve,r ,\sigma}.
\end{align*}
By definition of the norms and from (\ref{ipo1}), we have
$$
\|b_{ij}\partial_{ij}\phi\|_{\ve,r ,\sigma}+\|b_{i}\partial_{i}\phi\|_{\ve,r ,\sigma}
\leq C\delta \left(\| D^2_\xi \phi \|_{\ve , r , \sigma } + \| D_\xi \phi
\|_{\ve , r -1 , \sigma } +\|\phi \|_{\ve, r- 2 , \sigma}\right).
$$
Therefore, taking $\delta>0$ small enough, we  get (\ref{est0a}). Also we  get (\ref{est1a}) as a consequence of (\ref{est1}).

\medskip

Next we prove the existence of the solution
$\phi$ to problem \equ{eq:eqw}.
To this purpose  we consider the Hilbert space $\mathcal{H}$ defined as the subspace of functions $\psi$ which are in $H^1(\DD)$ such that $\psi=0$ on $\partial\hat\DD$, and
$$
\int_{\hat\DD} \psi (\rho z, \xi ) Z_j (\xi ) \, d\xi = 0  \  \foral z \in K_\e, \quad  j=0, \ldots N+1.
$$
Define a bilinear form in $\mathcal{H}$ by
$
B(\phi,\psi):=\int_{\hat\DD}\psi L\phi.
$
Then problem (\ref{eq:eqwd}) gets weakly formulated as that of finding $\phi\in \mathcal{H}$ such that
$
B(\phi,\psi)=\int_{\hat\DD}h\psi\quad \forall\ \psi\in \mathcal{H}.
$
By the Riesz representation theorem, this is equivalent to solve
$
\phi= T(\phi)+\tilde{h}
$
with $\tilde{h}\in \mathcal{H}$ depending linearly on $h$, and $T:  \mathcal{H} \rightarrow \mathcal{H}$ being a compact operator.
Fredholm's alternative guarantees that there is a unique solution to problem (\ref{eq:eqwd}) for any $h$ provided that
\begin{eqnarray}\label{linear7}
\phi= T(\phi)
\end{eqnarray}
has only the zero solution in $\mathcal{H}$. Equation (\ref{linear7}) is equivalent to problem (\ref{eq:eqwd}) with $h=0$.
If $h=0$, the estimate in (\ref{est0a}) implies that $\phi=0$.
This concludes the proof of Proposition \ref{linear}.
\end{proof}

Now we show how we can construct  very accurate approximate solutions to
Problem \eqref{adesso} locally close to $K_\rho$,  using  an iterative method that we describe below: let $I$ be an integer.
The
expanded variables $(z, \xi)$ will be defined as in \equ{b0} with
\begin{equation}\label{deffmu}
\mu_\ve(y)=  \mu_{0,\ve}+\mu_{1,\e}+\dots +\mu_{I,\e},\quad y=\rho z
\end{equation}
where $\mu_{0,\ve} , \, \mu_{1,\e} , \ldots , \mu_{I , \e} $ will be
smooth functions on $K$, with
$
\mu_{0,\ve}=\mu_0+\ve^{\frac{1}{N-2}}\bar{\mu}_0$, $\mu_0>0
$ as defined in (\ref{defmu0dn0}).
Moreover
\begin{equation}\label{deffPhi}
d_{j,\ve}(y) =  d^{0}_{j,\ve}+d^{1}_{j,\ve} +\dots+d^{I}_{j,\ve} ,\ \ \ j=1,\dots,N,
\end{equation}
where $d_{j,\e}^{\ell}$, $j=1,\dots,N; \, \ell=0,\ldots,I$, will be smooth
functions defined along $K$ with values in $\R$, with
$
d^0_{N,\ve}=d_{N}^0+\ve^{\frac{1}{N-2}}\bar{d}_{N}^0$, $d_{N}^0>0$ as defined in (\ref{defmu0dn0}).
In the $(z, \xi )$ variables, the shape of the approximate solution will be  given by
\begin{equation}\label{vIe1}
v_{I+1,\e}(z, \ov \xi, \xi_N)= \tilde{\omega}_{I+1,\ve}+ \tilde{e}_\ve (y ) \chi_\ve (\xi)
Z_0,\quad y=\rho z\in K,
\end{equation}
with
\begin{equation}\label{vIe}
\tilde{\omega}_{I+1,\ve}= U\left( \xi \right)-\bar{U}\left( \xi \right) +
  w_{1,\e}\left(z, \xi \right)
  + \dots +  w_{I+1,\e}\left(z, \xi \right), \quad \xi = (\bar \xi , \xi_N )
\end{equation}
where $\bar U$ is given by
\begin{equation}\label{ubar}
\bar{U}\left( \xi \right)=U\left( \bar{\xi},\xi_N+2\frac{ \tilde{d}_{N,\ve}}{\tilde{\mu}_\ve} \right)={\alpha_N \over (1+|\bar{\xi}|^2+|\xi_N+2\frac{\tilde{d}_{N,\ve}}{\tilde{\mu}_\ve}|^2)^{N-2 \over 2} }, \quad \alpha_N= (N(N-2))^{\frac{N-2}4},
\end{equation}
and  the functions $w_{j, \e}$'s for $j\ge 1$ are to be
determined so that the above function $v_{I+1,\e}$
satisfies formally
$$
\mathcal{S}_\ve(v_{I+1,\e})=-{\mathcal A}_{\mu_\ve , d_\e } v_{I+1,\e} -\mu_{\ve}^{\frac{N-2}{2}\ve}v_{I+1,\e}^{\frac{N+2}{N-2}-\ve}=\mathcal{O}(\ve^{I+2}) \quad {\mbox {in}} \quad K_\rho\times \hat\DD.
$$
In the second term in \equ{vIe1}, $Z_0$ denotes the first
eigenfunction in $L^2(\R^N)$  of the problem
$$
 \Delta \phi + pU^{p-1}\phi  = \la \phi\quad   \mbox{ in} \quad
 \R^{N}, \quad \lambda_1 >0
$$
with $\int Z_0^2 = 1$ and $\chi_\ve$ is a cut off function defined
as follows. Let $\chi = \chi (s) $, for $s \in \R$, with $\chi (s) =
1$ if $s<\hat \delta $, $\chi (s) = 0$ if $s > 2\hat \delta$, for
some fixed $ \hat \delta>0$ to be chosen in such a way that $\chi_\ve (\bar
\xi, -{\tilde{d}_{N,\ve} \over \tilde{\mu}_\ve } ) =0 $, where $\chi_\ve (\xi) = \chi
(\ve^{1\over N-2} |\xi| )$. Observe that the function $v_{I+1}$ satisfies
the Dirichlet boundary condition for $\xi_N = - {\tilde{d}_{N,\ve} \over
\tilde{\mu}_\ve } $.

Finally, in \equ{vIe1} the function $\tilde{e}_\ve (\rho z )$ is defined as
follows
\begin{equation}\label{canuto111}
\tilde{e}_\ve =\ve e_\ve= \ve (e_{0}+e_{1,\ve}+\cdots+e_{I,\ve})
\end{equation}
where
$
e_{0,\ve}=e_0+\ve^{\frac{1}{N-2}}\bar{e}_0,
$
with $e_0$ is
an explicit smooth function, uniformly bounded in $\ve$, whose expression is given in (\ref{defe0}).

Next Proposition shows existence and qualitative properties of the functions $\mu_\e$, $d_\e$ and
$v_{I+1, \e}$ as described above.
We prove the following result.

\begin{proposition} \label{Construction}
For any integer $I \in \N$ there exist  smooth functions $\mu_\e :K
\to \R$ and $d_{1,\e},\ldots,d_{N,\ve} : K \to \R^{N }$, $e_\ve: K \to \R$,
such that
\begin{equation}
\label{bf2} \| \, \mu_{\e}\|_{L^\infty (K)} +\| \partial_a
\mu_{\e}\|_{L^\infty (K)} +\|\partial^2_a \, \mu_{\e}\|_{L^\infty
(K)} \leq C
\end{equation}
\begin{equation}
\label{bf3} \| \, d_{j,\e}\|_{L^\infty (K)} +\| \partial_a
d_{j,\e}\|_{L^\infty (K)} +\|\partial^2_a
d_{j,\e}\|_{L^\infty (K)} \leq C,\quad \mbox{for}\ j=1,\ldots,N,
\end{equation}
\begin{equation}\label{bf45}
\|e_\ve\|_{L^\infty (K)} +\| \partial_a
e_\ve\|_{L^\infty (K)} +\|\partial^2_a
e_\ve\|_{L^\infty (K)} \leq C
\end{equation}
for some positive constant $C$, independent of $\e$.  Moreover there exists a positive
function $v_{I+1, \e} :K_\rho\times \hat\DD \to \R$ such that
$$
-{\mathcal A}_{\mu_\e , d_\e} (v_{I+1 , \e} ) -\mu_\ve^{\frac{N-2}{2}\ve}v_{I+1, \e}^{p-\ve} = {\mathcal E}_{I+1 , \e} \quad {\mbox {in}}
\quad  \DD
$$
$$
v_{I+1 , \e}  = 0 \quad {\mbox {on}}
\quad \partial \DD
$$
with
\begin{equation} \label{bf4}
\| v_{I+1 , \e} - v_{I , \e} \|_{\e , 2 , \sigma } \leq C
\e^{I+1}, \quad \| {\mathcal E}_{I+1 , \e} \|_{\e , 4 , \sigma} \leq C \e^{I+2}.
\end{equation}
\end{proposition}

To construct  such accurate approximate  solutions, we use an iterative scheme of Picard's type.
The arguments have been used in previous works, but in turns out that in this paper the arguments are more involved. For this reason we give a full detailed construction here.  This is the aim of the rest of this section.

\subsubsection*{ \bf Construction of $w_{1,\e}$, $\mu_{0,\ve}$, $d^0_{N,\ve}$ and $e_{0,\ve}$ : }
For this first step we define
$$
v_{1 , \e} = U-\bar{U} + w_{1,\e}+\ve e_{0,\ve}\chi_\ve(\xi) Z_0
$$
with
$
\mu_\ve = \mu_{0,\ve} ,\quad d_{N,\ve} = d^0_{N,\ve},\quad \hbox{and } \quad e_{\ve}=e_{0,\ve}.
$
Using the expansion of $S_\ve(v_{1,\ve})$ given in Lemma
\ref{scaledlaplacianco} with  $U=w_N$ is the standard bubble defined in \eqref{wn},
we then have, in $\DD$,
\begin{align*}
\mathcal{S}_\ve (v_{1,\e})
 =&- {\mathcal A}_{\mu_\ve , d_\e } \left(U-\bar{U} + w_{1,\e}+\ve  e_{0,\ve} Z_0\right) -\mu_{0,\ve}^{\frac{N-2}{2}\ve}\left(U-\bar{U} + w_{1,\e}+\ve  e_{0,\ve} Z_0\right)^{p-\ve}\\[3mm]
  = &- \D_{\R^{N}} w_{1,\e}
  - p U^{p-1} w_{1,\e} -2(\ve\,   d_{N,\ve}^0 + \rho\,  \mu_{0,\ve} \xi_N)H_{ij}\partial_{ij}   w_{1,\ve}   +\rho\mu_{0,\ve} H_{\alpha\alpha}\partial_N
   w_{1,\ve} \\[3mm]
 &+ \bar U^p + pU^{p-1}\bar{U} +\ve \left\{ U^p\log U-\frac{N-2}{2}U^p \log(\mu_{0,\ve})-2d^0_{N,\ve}H_{ij}\partial_{ij}^2U-\lambda_1 e_{0,\ve} Z_0\right\} \nonumber\\
 & -\rho\mu_{0,\ve}  \left[2\xi_NH_{ij}\partial_{ij}^2 U-H_{\alpha\alpha}\partial_NU\right]   + \ve^2\mathcal{E}_{1,\e}  + Q_\ve (w_{1,\ve })\nonumber\\
 =& \mathcal{L}_{\ve}w_{1,\ve}+h_{1,\ve}  + \ve^2\mathcal{E}_{1,\e}  + Q_\ve (w_{1,\ve }),
\end{align*}
where the operator $\mathcal{L}_{\ve}$ is given by
\begin{equation} \label{rivoli3}
\mathcal{L}_{\ve}w_{1,\e} := - \D_{\R^{N}} w_{1,\e}
  - p U^{p-1} w_{1,\e} -2(\ve\,   d_N + \rho\,  \mu_\ve \xi_N)H_{ij}\partial_{ij}   w_{1,\ve}   +\rho\mu_\ve tr(H)\partial_N
   w_{1,\ve}.
\end{equation}
The term $h_{1,\e} $ is defined as follow
\begin{align}\label{g1e}
h_{1,\e}   =& pU^{p-1}\bar{U} +\ve \left\{U^p\log U -\frac{N-2}{2}U^p \log(\mu_{0,\ve})- 2d^0_{N,\ve}H_{ij}\partial_{ij}^2U-\lambda_1 e_{0,\ve} Z_0\right\} \nonumber\\
  &-\rho \mu_{0,\ve}  \left[2\xi_NH_{ij}\partial_{ij}^2 U-H_{\alpha\alpha}\partial_NU\right].
\end{align}
The function $\mathcal{E}_{1, \e}$ is a function which is a sum of
functions of the form
\begin{eqnarray}\label{youjieas}
&&f_1(\rho z) \left[ f_2 ( \mu_{0,\ve}, d^0_{N,\ve},e_{0,\ve}, \partial_a\mu_{0,\ve}, \partial_ad^0_{N,\ve},\partial_e e_{0,\ve}) +\right.\nonumber\\
 &&  \left.+o(1) f_3 ( \mu_{0,\ve}, d^0_{N,\ve},e_{0,\ve},  \partial_a\mu_{0,\ve}, \partial_ad^0_{N,\ve}, \partial_ae_{0,\ve} , \partial^2_{aa}\mu_{0,\ve},
\partial^2_{aa}d^0_{N,\ve}  , \partial^2_{aa}e_{0,\ve} ) \right] f_4(y)
\end{eqnarray}
with $f_1$ a smooth function uniformly bounded in $\ve$, $f_2$ and
$f_3 $ are smooth functions of their arguments, uniformly bounded in
$\ve$ as $\mu_{0,\ve}$, $d^0_{N,\ve}$ and $e_{0,\ve}$ are uniformly bounded. An
important remark is that the function  $f_3$ depends linearly on the
argument. Concerning $f_4$, we have
$$
\sup (1+ |\xi|^4 )|f_4(y)| < + \infty.
$$
The term $Q_\ve (w_{1,\e} )$ is quadratic in $w_{1,\e}$, in fact it
is explicitly given by
$$
\mu_{0,\ve}^{  \frac{N-2}{2}\e}\left[(U-\bar{U} + w_{1\e}+\ve  e_{0,\ve} Z_0 )^{p- \e} - U^{p\pm \e} - p U^{p-1\pm \e} (\bar{U}+w_{1,\e}+\ve  e_{0,\ve} Z_0)\right].
$$
We now define $\mu_\ve = \mu_{0,\ve}$ , $d_{N,\ve} = d^0_{N,\ve}$, and $e_{\ve}=e_{0,\ve}$ in such a way that
\begin{equation}\label{proh1zl}
\int_{\hat\DD}h_{1,\ve}Z_{l}d\xi = 0\qquad \hbox{ for all }\quad  l=0, 1,\ldots,N.
\end{equation}
Since  $h_{1,\ve}$ is an even function on the variable $\bar \xi$ (due to the fact that $U$ and $\bar U$ are even in $\bar\xi$) since the set $\hat\DD$ is symmetric in the variable $\bar \xi$, the above condition is authomatically satisfied for any $l=1, \ldots , N-1$.

On the other hand, we have (see Section \ref{secapp} for a proof)
\begin{equation}\label{proh1zn1}
 \int_{\hat\DD}h_{1,\ve}Z_{N+1}d\xi
  =  \ve \left[-
A_1 \left( { \mu_{0,\ve} \over   d_{N,\ve}^0} \right)^{N-2} +  A_2  +
\ve^{1\over N-2}
  \left({\mu_{0,\ve} \over   d_{N,\ve}^0} \right)^{N-1} g_{N+1} \left( {\mu_{0,\ve} \over   d_{N,\ve}^0} \right) \right](1+o(1)),
\end{equation}
\begin{equation}\label{proh1zn}
\int_{\hat\DD}h_{1,\ve}Z_{N}d\xi  =  \ve^{1+\frac{1}{N-2}} \left[A_3
  \left( { \mu_{0,\ve} \over   d_{N,\ve}^0} \right)^{N-1} +A_6\mu_{0,\ve} \, H_{aa} +
\ve^{1\over N-2}
  \left({\mu_{0,\ve} \over   d_{N,\ve}^0} \right)^{N} g_{N} \left( {\mu_{0,\ve} \over   d_{N,\ve}^0} \right)\right](1+o(1)),
\end{equation}
and
\begin{eqnarray}\label{proh1z0h}
\int_{\hat\DD}h_{1,\ve}Z_{0}d\xi &=& \ve  \left[A_4
  \left( { \mu_{0,\ve} \over   d_{N,\ve}^0} \right)^{N-2} +  A_5-A_7\log(\mu_{0,\ve})-\lambda_1 e_{0,\ve}-2H_{jj}d_{N,\ve}^0\int_{\R^N}\partial^2_{jj}UZ_0d\xi\right.\nonumber\\
  &&\quad\left.+
\ve^{1\over N-2}
  \left({\mu_{0,\ve} \over   d_{N,\ve}^0} \right)^{N-1} g_{0}\left( {\mu_{0,\ve} \over   d_{N,\ve}^0} \right)\right](1+o(1))
\end{eqnarray}
where the functions $g_i$ are smooth function
with $g_i (0) \not= 0$ and $A_i$ are positive constants.

\medskip
Let $(  \mu_{0,\ve} , d_{N,\ve}^0,  e_{0,\ve} ) : K \to \in
(0,\infty) \times (0,\infty)\times \R$ be the solution to the
following system of nonlinear algebraic equations
\begin{equation}\label{systemC}
\left\{\begin{array}{lll}
-A_1   \left ( \frac{\mu_{0,\ve}}{d_{N,\ve}^0} \right )^{N-2} +A_2
    + \ve^{\frac 1{N-2}} \,
 \left ( \frac{\mu_{0,\ve}}{d_{N,\ve}^0}\right )^{N-1}
g_{N+1} \left ( \frac{\mu_{0,\ve}}{d_{N,\ve}^0}\right )\,  =\ 0
\cr \\[3mm]
  A_1
 \left ( \frac{\mu_{0,\ve}}{d_{N,\ve}^0}\right )^{N-1}+\frac{A_1A_6}{A_3} \mu_{0,\ve} H_{aa}
    + \ve^{\frac 1{N-2}} \,
 \left ( \frac{\mu_{0,\ve}}{d_{N,\ve}^0}\right )^{N}   g_{N} \left ( \frac{\mu_{0,\ve}}{d_{N,\ve}^0}\right )\, = 0
 \cr \\[3mm]
 A_4\left(\frac{\mu_{0,\ve}}{d_{N,\ve}^0}\right)^{N-2}+A_5-A_7\log(\mu_{0,\ve})-
\lambda_1 e_{0,\ve}  \\[1mm]
\qquad - 2H_{jj}  d_{N,\ve}^0\int_{\R^N}\partial_{jj}^2 UZ_0d\xi
     + \ve^{\frac 1{N-2}} \,
 \left ( \frac{\mu_{0,\ve}}{d_{N,\ve}^0}\right )^{N-1}    g_{0} \left ( \frac{\mu_{0,\ve}}{d_{N,\ve}^0}\right )\, = \, 0 \ .
\end{array}\right.
\end{equation}
This solution $( \mu_{0,\ve}  , d_{N,\ve}^0 ,
e_{0,\ve} )$ exists and has the form
\begin{equation}\label{choiceofmude}
\mu_{0,\ve} = \mu_0 + \ve^{1\over N-2} \bar{\mu}_0, \quad  d_{N,\ve}^0 = d^0_{N} +
\ve^{1\over N-2} \bar{d}^0_{N}, \quad   e_{0,\ve}   = e_0 + \ve^{1\over N-2} \bar{e}_0,
\end{equation}
where $\mu_0$, $d^0_{N}$ and $e_0$ solve
\begin{equation}\label{soluofmude}
F(\mu_0 , d^0_N , e_0 ) = 0
\end{equation}
where
$$
 F(\mu, d_N , e ) := \left(
 \begin{array}{c}
 -A_1   \left ( \frac \mu
{d_N} \right )^{N-2} + A_2 \ \\
 A_1
 \left ( \frac \mu {d_N} \right )^{N-1}+\frac{A_1A_6}{A_3} \mu H_{aa}\\
 A_4\left(\frac{ \mu }{ d_{N}}\right)^{N-2}+A_5-A_7\log(\mu)-
\lambda_1 e   - 2H_{jj} d_{N}\int_{\R^N}\partial_{jj}^2 UZ_0d\xi\\
\end{array}
\right)
 $$
Explicitely,  we get
\begin{equation}\label{defmu0dn0}
\mu_0=-\left(\frac{A_2}{A_1}\right)^{\frac{N-1}{N-2}}\frac{A_3}{A_6}\frac{1}{H_{aa}},\qquad d^0_{N}=-\frac{A_2}{A_1}\frac{A_3}{A_6}\frac{1}{H_{aa}}
\end{equation}
and
\begin{equation}\label{defe0}
e_0=\frac1{\lambda_1}\bigg\{-2 d_{N}^0 H_{jj}\int_{\R^N}\partial_{jj}^2 UZ_0d\xi+\frac{A_2A_4}{A_1}+A_5-A_7\log(\mu_{0})\bigg\}.
\end{equation}
Exactly at this point is where  we need to assume  that the mean curvature in the directions of  $T_qK$ is negative for any $q\in K$ in order to ensure that $\mu_0$ is positive.

Direct computations give
\begin{eqnarray*}
F_0 :=  \nabla_{\mu , d_N , e} F(\mu_0 ,d^0_{N} , e_0 )
=\left(
\begin{array}{ccc}
-(N-2) A_1
{\mu_0^{N-3} \over
(d^0_{N})^{N-2}} &
(N-2) A_1 {\mu_0^{N-2} \over
(d^0_{N})^{N-1}} & 0 \\[3mm]
 (N-2) A_1 {\mu_0^{N-2} \over
(d^0_{N})^{N-1}}&
-(N-1) A_1 {\mu_0^{N-1} \over
(d^0_{N})^{N}} & 0 \\[3mm]
 a_{31}&
 a_{32} & -\lambda_1 \\
\end{array} \right),
\end{eqnarray*}
where
$$
a_{31}=(N-2)A_4 \frac{\mu_0^{N-3} }{ (d^0_{N})^{N-2}} -\frac{A_7}{\mu_0},\qquad
a_{32}=-(N-2)A_4 \frac{\mu_0^{N-2} }{ (d^0_{N})^{N-1}}-2 H_{jj}d^0_{N}\int_{\R^N}\partial_{jj}^2UZ_0d\xi.
$$
Since
$$
\mbox{det} \left( F_0 \right) =-
\lambda_1 (N-2) A_1^2 {\mu_0^{N-2} \over
(d^0_{N})^{N-1}}
H_{aa}>0,
$$
solving system \equ{systemC} is equivalent to solve a fixed point
problem, which is uniquely solvable in the set
$$
\bigg\{ (\bar{\mu}_0 , \bar{d}_{N}^0 , \bar{e}_0 ) \, : \, \| \bar{\mu}_0 \|_\infty \leq \delta,\  \| \bar{d}_{N}^0\|_\infty \leq \delta,\  \|
\bar{e}_0 \|_\infty \leq \delta \bigg\}
$$
for some proper small $\delta>0$. Moreover, the smoothness of $\bar{\mu}_0 , \bar{d}_{N}^0 , \bar{e}_0$ follows using of the Implicit function Theorem.

For a later purpose we define the following quantities which appeared in the above matrix $F_0$
$$
A:=-(N-2) A_1
{\mu_{0}^{N-3} \over
(d^0_{N})^{N-2}} , \quad
B= (N-2) A_1 {\mu_{0}^{N-2} \over
(d^0_{N})^{N-1}}, \quad
C= -(N-1) A_1 {\mu_{0}^{N-1} \over
(d^0_{N})^{N}}.
$$
An easy computation shows that $AC-B^2 >0$.

\medskip


\medskip
With the choice for $\mu_{0,\ve}$, $d^0_{N,\ve}$ and $e_{0,\ve}$ in (\ref{choiceofmude}), the integral of  the right hand side
in \eqref{eq:eqw1} against $Z_{l}$, $l=0,1,\ldots,N+1$, vanishes on $\hat\DD$.
Furthermore, with this choice of  $\mu_{0,\ve}$, $d^0_{N,\ve}$ and $e_{0,\ve}$ in (\ref{choiceofmude}), the linear operator $\mathcal{L}_{\ve}$ defined in \eqref{rivoli3} satisfies the assumptions of Proposition \ref{linear}. Thus, we  define $w_{1,\e}$ to be solution of the Problem
\begin{equation}\label{eq:eqw1}
    \mathcal{L}_{\ve}w_{1,\e} =- h_{1,\e}  \quad\hbox{ in } \DD\qquad
     w_{1,\e}  = 0, \ \, \quad\hbox{ on } \partial\DD.
\end{equation}
Moreover, it is straightforward to check that
$$
\|h_{1,\e} \|_{\ve , 4 , \sigma } \leq C\ve
$$
for some $\sigma \in (0,1)$. Proposition \ref{linear} thus gives that
\begin{equation}
\label{ew1} \| D^2_\xi w_{1,\e} \|_{\ve , 4 , \sigma } + \| D_\xi
w_{1,\e} \|_{\ve , 3 , \sigma } +\|w_{1,\e} \|_{\ve, 2,
\sigma}\le C \ve
\end{equation}
and that there exists a positive constant $\beta$ (depending only on
$\Omega, K$ and $N$) such that for any integer $\ell$ there holds
\begin{equation}\label{eq:estw1}
    \|\nabla^{(\ell)}_{z} w_{1,\e}(z,\cdot)\|_{\ve,2,\sigma} \leq \beta C_l
    \ve \qquad \quad z \in K_\rho
\end{equation}
where $C_l$ depends only on $l$, $p$, $K$ and $\Omega$.

\medskip
\noindent
With this definition of $\mu_{0,\ve}$, $d^0_{N,\ve}$,  $e_{0,\ve}$ and $w_{1, \e}$, we have in particular that
$$
\| - {\mathcal A}_{\mu_\ve , d_\e } v_{1,\e}  - \mu_\ve^{  \frac{N-2}{2}\ve }\,v_{1,\e}^{p- \ve} \|_{\e , 4 , \sigma} \leq C \e^2.
$$

\medskip

\subsubsection*{\bf  Construction of $w_{2,\e}$ and choice of the parameters $\mu_{1,\e}$, $d^1_{N,\e}$ and $e_{1,\ve}$}

To improve further our approximate solutions $v_{1,\e}$ constructed in the previous step we define  the function
$$v_{2,\e}(z, \xi)= U(\xi)-\bar{U}\left( \xi \right) +
  w_{1,\e}\left(z, \xi \right)
  +   w_{2,\e}\left(z, \xi \right)+\ve e_\ve \chi_\ve(\xi)Z_0,$$
  where now
  $\mu_\e= \mu_{0,\ve}+\mu_{1,\e}$, $d_{N,\ve} =  d^0_{N,\ve}  +d^1_{N, \e}$, $e_\ve=e_{0,\ve}+e_{1,\ve}$ and where $\mu_{0,\ve}$, $d^0_{N,\ve}$, $e_{0,\ve}$ and $w_{1,\ve}$ have already been constructed  in the previous step.
  Observe that a Taylor expansion yields
\begin{align}\label{baruexp}
\bar U(\xi) =&U\left( \bar{\xi},\xi_N+2\frac{ \ve (d^0_{N,\ve}  +d^1_{N, \e})}{\rho (\mu_{0,\ve}+\mu_{1,\e})} \right)=U\big( \bar{\xi},\xi_N+2\frac{ \ve d^0_{N,\ve}}{\rho \mu_{0,\ve}} \big)\nonumber\\
 +& 2\frac\e\rho\partial_NU\big( \bar{\xi},\xi_N+2\frac{ \ve d^0_{N,\ve}}{\rho \mu_{0,\ve}} \big)\left\{ \frac{  d^0_{N,\ve}   }{ \mu_{0,\ve} }\left(\frac{d^1_{N, \e}}{d^0_{N, \e}}-\frac{\mu_{1,\e}}{\mu_{0,\e}}\right)+O\left(\frac{d^1_{N, \e}}{d^0_{N, \e}}-\frac{\mu_{1,\e}}{\mu_{0,\e}}\right)^2\right\}.
\end{align}

Computing ${\mathcal S}_\e (v_{2,\e})$ (see \equ{Sep}) we get
\begin{eqnarray}\label{rivoli4}
{\mathcal S}_\e (v_{2,\e})=
\mathcal{L}_{\e} w_{2, \e}+ h_{2, \e}+  \ve^3 \mathcal{E}_{2,\e}  +Q_\ve (w_{2,\e} )
\end{eqnarray}
where $\mathcal{L}_{\e}$ is defined in \equ{rivoli3} and  the function $h_{2,\e}$ is given by
\begin{align}\label{w2epsilon}
 h_{2,\ve } =  & -2\ve d^1_{N,\ve}  H_{ij}\partial_{ij}^2U +\rho \, \mu_{1,\e}\ \left[-2\xi_NH_{ij}\partial_{ij}^2U+H_{\alpha\alpha}\partial_NU\right]-\lambda_1\,\e\,e_{1,\e}\, Z_0\nonumber\\
 &-\ve\frac{N-2}{2}\frac{\mu_{1,\e}}{\mu_{0,\e}}U^p+ \tilde{f}_{2\ve} +\tilde{h}_{2\ve}(y,\xi,\mu_{0,\ve},d^0_{N,\ve},e_{0,\ve})
\end{align}
where
$$
\tilde{f}_{2\ve}=2pU^{p-1}\partial_N U\big( \bar{\xi},\xi_N+2\frac{ \ve d^0_{N,\ve}}{\rho \mu_{0,\ve}} \big)\frac{ \ve  d^0_{N,\ve}   }{\rho \mu_{0,\ve} }\left[\frac{d^1_{N, \e}}{d^0_{N, \e}}-\frac{\mu_{1,\e}}{\mu_{0,\e}}\right],
$$
and $\tilde{h}_{2\ve}$ is a smooth function on its variables and is even in the variable $\bar \xi\in \R^{N-1}$, which  implies in particular that
\begin{eqnarray}\label{tildef2epro1}
\int_{\hat\DD} \tilde{h}_{2\ve} Z_jd\xi=0
\quad j=1,\ldots,N-1,
\end{eqnarray}
Moreover we can easily show that
\begin{eqnarray}\label{tildef2epro10}
\int_{\hat\DD} \tilde{h}_{2\ve} Z_0d\xi=\ve^2\vartheta_1(\mu_{0,\ve},d_{N,\ve}^0,e_{0,\ve}), \quad
\int_{\hat\DD} \tilde{h}_{2\ve} Z_{N+1}d\xi= \ve^2\vartheta_2(\mu_{0,\ve},d_{N,\ve}^0,e_{0,\ve}),
\end{eqnarray}
and
\begin{eqnarray}\label{tildef2epro3}
\int_{\hat\DD} \tilde{h}_{2\ve} Z_N d\xi=\ve \rho\vartheta_3(\mu_{0,\ve},d_{N,\ve}^0,e_{0,\ve}).
\end{eqnarray}
where where $\vartheta_i(\mu_{0,\ve},d_{N,\ve}^0,e_{0,\ve}), i=1,2,3$ are some uniformly bounded functions.
In \equ{rivoli4} the term $ \mathcal{E}_{2,\e}$ can be described as
the sum of functions of the form (\ref{youjieas}).
Finally the term $Q_\ve (w_{2,\e} )$ is a sum of quadratic terms in
$w_{2,\e}$ like
$$
(\mu_{0,\ve}+\mu_{1,\ve})^{\frac{N-2}{2}\ve}\left[-(U-\bar{U} + w_{1,\e} + w_{2,\e}  +\ve e_\ve \chi_\ve(\xi)Z_0)^{p-\ve} \right.
$$
$$\left.+ (U-\bar{U} + w_{1,\e} +\ve e_\ve \chi_\ve(\xi)Z_0)^{p-\ve}  +(p-\ve) (U-\bar{U} +
w_{1,\e}+\ve e_\ve \chi_\ve(\xi)Z_0 )^{p-1-\ve}  w_{2,\e}\right]
$$
and linear terms in $w_{2,\e}$ multiplied by a term of order
$\ve$, like
$$
(p-\ve) \left( (U-\bar{U} + w_{1,\ve} +\ve e_\ve \chi_\ve(\xi)Z_0)^{p-1-\ve} - U^{p-1-\ve} \right) w_{2,\ve}.
$$

\medskip
First  we define  $\mu_{1,\ve}, d^1_{N,\ve},e_{1,\ve}$.
Similar computations as in \equ{proh1zn1}-\equ{proh1z0h} yields
$$
\int_{\hat\DD}h_{2\ve}Z_{N+1}d\xi
  =  -\ve A_1 \left( { \mu_{0,\ve} \over   d_{N,\ve}^0} \right)^{N-2}\left[
(N-2)\left(\frac{\mu_{1,\ve}}{\mu_{0,\ve}}-\frac{d_{N,\ve}^1}{d_{N,\ve}^0}\right)  +
O\left(\ve^{1\over N-2}\right) \right](1+o(1))
$$
$$
\int_{\hat\DD}h_{2\ve}Z_{N}d\xi  =  \ve^{1+\frac{1}{N-2}} A_3(N-1)
  \left( { \mu_{0,\ve} \over   d_{N,\ve}^0} \right)^{N-1} \left(\frac{\mu_{1,\ve}}{\mu_{0,\ve}}-\frac{d_{N,\ve}^1}{d_{N,\ve}^0}\right)(1+o(1))
$$
$$
\int_{\hat\DD}h_{2\ve}Z_0d\xi =  \ve A_4 \left( { \mu_{0,\ve} \over   d_{N,\ve}^0} \right)^{N-2}\left[
(N-2)\left(\frac{\mu_{1,\ve}}{\mu_{0,\ve}}-\frac{d_{N,\ve}^1}{d_{N,\ve}^0}\right)  +
O\left(\ve^{1\over N-2}\right) \right](1+o(1)).
$$
We choose $\mu_{1,\ve}, d^1_{N,\ve},e_{1,\ve}$ so that
\begin{eqnarray}\label{tildef2epro1k}
\int_{\hat\DD} h_{2,\ve} Z_{l}d\xi=0,\quad l=0, N,N+1.
\end{eqnarray}
We can easily see that the above orthogonality conditions are fulfilled provided we choose the parameters  $\mu_{1,\ve}, d_{N,\ve}^1, e_{1,\ve}$ to solve the following system
\begin{equation}\label{systemCtgh}
\left\{\begin{array}{llll}
(N-2)A_1\frac{\mu_{0,\ve}^{N-3}}{(d_{N,\ve}^0)^{N-2}}\mu_{1,\ve}- (N-2)A_1\frac{\mu_{0,\ve}^{N-2}}{(d_{N,\ve}^0)^{N-1}}d_{N,\ve}
\cr
+ \ve^{\frac 1{N-2}}g_{N+1} \left ( \frac{\mu_{1,\ve}}{d_{N,\ve}^1}\right )=  \ve \Re_1(\mu_{0,\ve},d_{N,\ve}^0,e_{0,\ve}) \cr \\[3mm]
 \left[A_6H_{\a\a}-(N-1)A_3\frac{\mu_{0,\ve}^{N-2}}{(d_{N,\ve}^0)^{N-1}}\right]\mu_{1,\ve}+(N-1)A_3\frac{\mu_{0,\ve}^{N-1}}{(d_{N,\ve}^0)^{N}}d_{N,\ve}^1
\\[3mm] + \ve^{\frac 1{N-2}}g_{N} \left ( \frac{\mu_{1,\ve}}{d_{N,\ve}^1}\right )=
\ve\Re_2(\mu_{0,\ve},d_{N,\ve}^0,e_{0,\ve})
 \cr \\[3mm]
(N-2)A_4\frac{\mu_{0,\ve}^{N-2}}{(d_{N,\ve}^0)^{N-2}}\left(\frac{\mu_{1,\ve}}{\mu_{0,\ve}}-\frac{d_{N,\ve}^1}{d_{N,\ve}^0}\right)
+A_5-A_7\log(\mu_{1,\ve})-
\lambda_1 e_{1,\ve}  \\[3mm]
- 2H_{jj}  d_{N,\ve}^1\int_{\R^N}\partial_{jj}^2 UZ_0d\xi
     + \ve^{\frac 1{N-2}}g_{0} \left ( \frac{\mu_{1,\ve}}{d_{N,\ve}^1}\right )=\ve\Re_3(\mu_{0,\ve},d_{N,\ve}^0,e_{0,\ve}) \ ,
\end{array}\right.
\end{equation}
where $\Re_i$, $i=1,2,3$ are some smooth uniformly bounded functions.
Arguing as in the first step we can show that the above system is solvable and the solution $( \mu_{1,\ve}  , d_{N,\ve}^1 ,
e_{1,\ve} )$ has the form
\begin{equation}\label{choiceofmudea}
\mu_{1,\ve} = \tilde{\mu}_{1,\ve} + \ve^{1\over N-2} \overline{\mu}_{1,\ve}, \quad  d_{N,\ve}^1 = \tilde{d}^1_{N,\ve} +
\ve^{1\over N-2} \overline{d}^1_{N,\ve}, \quad   e_{1,\ve}   = {\tilde{e}}_{1,\ve} + \ve^{1\over N-2} \overline{e}_{1,\ve},
\end{equation}
where $(\tilde{\mu}_{1,\ve}$, $\tilde{d}^1_{N,\ve},\tilde{e}_{1,\ve})$ is a solution of
\begin{equation}\label{systemCtghja}
\ \left\{\begin{array}{llll}
\tilde{\mu}_{1,\ve}-  \frac{\mu_{0,\ve} }{ d_{N,\ve}^0 }\tilde{d}_{N,\ve}^1
=  \ve\tilde{\Re}_1(\mu_{0,\ve},d_{N,\ve}^0,e_{0,\ve}) \cr \\[3mm]
 \left[A_6H_{\a\a}-(N-1)A_3\frac{\mu_{0,\ve}^{N-2}}{(d_{N,\ve}^0)^{N-1}}\right]\tilde{\mu}_{1,\ve}
 +(N-1)A_3\frac{\mu_{0,\ve}^{N-1}}{(d_{N,\ve}^0)^{N}}\tilde{d}_{N,\ve}^1
 = \ve\Re_2(\mu_{0,\ve},d_{N,\ve}^0,e_{0,\ve})
 \cr \\[3mm]
(N-2)A_4\frac{\mu_{0,\ve}^{N-2}}{(d_{N,\ve}^0)^{N-2}}\left(\frac{\tilde{\mu}_{1,\ve}}{\mu_{0,\ve}}-\frac{\tilde{d}_{N,\ve}^1}{d_{N,\ve}^0}\right)
+A_5-A_7\log(\tilde{\mu}_{1,\ve})-
\lambda_1 \tilde{e}_{1,\ve}  \\[3mm]
- 2H_{jj}  \tilde{d}_{N,\ve}^1\int_{\R^N}\partial_{jj}^2 UZ_0d\xi
     = \ve \Re_3(\mu_{0,\ve},d_{N,\ve}^0,e_{0,\ve}) \ ,
\end{array}\right.
\end{equation}
where $\tilde{\Re}_1=\frac{1}{(N-2)A_1}\frac{(d_{N,\ve}^0)^{N-2}}{\mu_{0,\ve}^{N-3}}\Re_1$. Indeed,
the first two equations in \eqref{systemCtgh} can be rewritten in the following form
\begin{equation}\label{systemCtghjeq}
M\left(
   \begin{array}{c}
     \tilde{\mu}_{1,\ve} \\[3mm]
     \tilde{d}_{N,\ve}^1 \\
   \end{array}
 \right)=
 \ve \left(
   \begin{array}{c}
     \tilde{\Re}_1(\mu_{0,\ve},d_{N,\ve}^0,e_{0,\ve}) \\[3mm]
     \Re_2(\mu_{0,\ve},d_{N,\ve}^0,e_{0,\ve}) \\
   \end{array}
 \right)
\end{equation}
with the matrix
\begin{equation}\label{systematrix}
M=\left(
          \begin{array}{cc}
            1 & -\frac{\mu_{0,\ve} }{ d_{N,\ve}^0 } \\[3mm]
             A_6H_{\a\a}-(N-1)A_3\frac{\mu_{0,\ve}^{N-2}}{(d_{N,\ve}^0)^{N-1}}  & (N-1)A_3\frac{\mu_{0,\ve}^{N-1}}{(d_{N,\ve}^0)^{N}} \\[3mm]
          \end{array}
        \right)
\end{equation}
which is clearly invertible since $\det(M)=A_6H_{\a\a}\,\frac{\mu_{0,\ve} }{ d_{N,\ve}^0 }\neq 0$. Thus we can get the existence of $\mu_{1,\ve}$ and $d_{N,\ve}^1$ in  (\ref{systemCtghja}), and we then get the existence of $e_{1,\ve}$ from the third equation in (\ref{systemCtghja}).
Moreover,
we have the following estimates
\begin{equation*}
\label{emu1e} \| \mu_{1,\e}\|_{L^\infty (K)} +\| \partial_a
\mu_{1,\e}\|_{L^\infty (K)} +\|\partial^2_a \mu_{1,\e}\|_{L^\infty
(K)} \leq C \ve,
\end{equation*}
\begin{equation*}
\label{emu1e} \| d^1_{N,\e}\|_{L^\infty (K)} +\| \partial_a
d^1_{N,\e}\|_{L^\infty (K)} +\|\partial^2_a d^1_{N,\e}\|_{L^\infty
(K)} \leq C \ve
\end{equation*}
and
\begin{equation*}
\label{e1e} \| e_{1,\e}\|_{L^\infty (K)} +\| \partial_a
e_{1,\e}\|_{L^\infty (K)} +\|\partial^2_a e_{1,\e}\|_{L^\infty
(K)} \leq C \ve.
\end{equation*}
Observe now the following
\begin{itemize}
\item   from (\ref{tildef2epro1}) and using the facts that $\partial_{jj}^2U$ is even with respect to $\bar\xi$,  we have
\begin{eqnarray}\label{tildef2epro1l}
\int_{\hat\DD} h_{2,\ve} Z_jd\xi=0,\quad j=1,\ldots,N-1.
\end{eqnarray}
\item given the choice of the parameters \eqref{choiceofmudea}, the linear operator defined in \eqref{rivoli4} by \eqref{rivoli3}, which  depends on $\mu_\e$, $d_{N,\e}$ and $e_\e$, satisfies the assumptions of Proposition \ref{linear}.
\end{itemize}
Henceforth, we apply the result of Proposition \ref{linear} to define   $w_{2,\e}$ to solve
\begin{equation}\label{edith2}
    \mathcal{L}_{\ve}w_{2,\e} =- h_{2,\e} \quad\hbox{ in } \DD\qquad
     w_{2,\e}  = 0, \ \,  \quad\hbox{ on } \partial\DD.
\end{equation}
Since,  for a given $\sigma \in (0,1)$,
$
\|   h_{2,\e} \|_{\e, 4, \sigma} \leq C \ve^2,
$
we have that
\begin{equation}
\label{ew2} \| D^2_\xi w_{2,\e} \|_{\ve , 4 , \sigma } + \| D_\xi
w_{2,\e} \|_{\ve , 3, \sigma } +\|w_{2,\e} \|_{\ve, 2 ,
\sigma}\le C \ve^{2}
\end{equation}
and that there exists a positive constant $\beta$ (depending only on
$\Omega, K$ and $n$) such that for any integer $\ell$ there holds
\begin{equation}\label{eq:estw2}
    \|\nabla^{(\ell)}_{y} w_{2,\e}(z,\cdot)\|_{\ve,2,\sigma} \leq \beta C_\ell \,
    \ve^{2} \qquad \quad \rho y = z \in K_\rho
\end{equation}
where $C_\ell$ depends only on $\ell$, $p$, $K$ and $\Omega$.

\medskip
With this choice of $\mu_{1,\e}$, $e_{1,\ve}$, $d^1_{N,\ve}$ and $w_{2, \e}$ we get that
$$
\| - {\mathcal A}_{\mu_\ve , d_\e } v_{2,\e}  - \mu_\ve^{ \frac{N-2}{2}\ve}\,v_{2,\e}^{p- \ve} \|_{\e, 4 , \sigma} \leq C \e^3.
$$

\medskip

\subsubsection*{\bf  Construction of $w_{3,\e}$ and choice of $\mu_{2,\e}$, $d^2_{N,\e}$, $e_{2,\ve}$ and $d^{0}_{j,\ve}$, $l=1,\ldots,N-1$ }

\

We define
$$v_{3,\e}(z, \xi)= U(\xi)-\bar{U}\left( \xi \right) +
  w_{1,\e}\left(z, \xi \right)
  +   w_{2,\e}\left(z, \xi \right)+   w_{3,\e}\left(z, \xi \right)+\ve e_\ve \chi_{\ve}(\xi) Z_0$$
  where $\mu_\e= \mu_{0,\ve}+\mu_{1,\e}+\mu_{2,\ve}$, $e_\ve=e_{0,\ve}+e_{1,\ve}+e_{2,\ve}$, $d_{N,\ve} =  d^0_{N,\e}+d^1_{N,\e}+d^2_{N,\e}$, $d_{l,\ve}=d^{0}_{1,\ve}$, $l=1,\ldots,N-1$. We remind that
 $\mu_{0,\ve},\mu_{1,\ve}$, $e_{0,\ve}, e_{1,\ve}$, $d^0_{N,\e}, d^1_{N,\e}$ and $w_{1,\ve}, w_{2,\ve}$ have already been constructed  in the previous steps.
Computing ${\mathcal S}_\e (v_{3,\e})$ (see \equ{Sep}) we get
\begin{equation}\label{rivoli43}
\mathcal{S}_\ve (v_{3,\e})  = \mathcal{L}_{\e} w_{3, \e}-  h_{3, \e}+ \ve^4 \mathcal{E}_{3,\e}  +   Q_\ve (w_{3,\e} )
\end{equation}
where $\mathcal{L}_{\e}$ is defined in \equ{rivoli3},  and  the function $h_{3,\e}$ is given by
\begin{align}\label{w2epsilon3}
h_{3,\ve } =&   -2\ve d^2_{N,\ve}  H_{ij}\partial_{ij}^2U +\rho \, \mu_{2,\e}\ \left\{-2\xi_NH_{ij}\partial_{ij}^2U+H_{\alpha\alpha}\partial_NU\right\}-\lambda_1\,\e\,e_{2,\e}\, Z_0-\ve\frac{N-2}{2}\frac{\mu_{2,\e}}{\mu_{0,\e}}U^p\nonumber\\
 &+ 2pU^{p-1}\partial_N U\big( \bar{\xi},\xi_N+2\frac{ \ve d^0_{N,\ve}}{\rho \mu_{0,\ve}} \big)\frac{ \ve  d^0_{N,\ve}   }{\rho \mu_{0,\ve} }\left[\frac{d^2_{N, \e}}{d^0_{N, \e}}-\frac{\mu_{2,\e}}{\mu_{0,\e}}\right]+\ve^2\rho\Xi_3(d^0_{j,\ve})\nonumber\\
 &+\tilde{h}_{3\ve}(y,\xi,\mu_{0,\ve},\mu_{1,\ve},d^0_{N,\ve},d^1_{N,\ve},e_{0,\ve},e_{1,\ve})
\end{align}
with the function $\tilde{h}_{3,\ve}$ satisfying
\begin{equation}\label{tildef2epro13}
\int_{\hat\DD} \tilde{h}_{3,\ve} Z_jd\xi= O(\ve^2\rho)  ,\quad j=1,\ldots,N-1,
\end{equation}
and
\begin{equation}\label{tildef2epro23}
\int_{\hat\DD} \tilde{h}_{3,\ve} Z_{N+1}d\xi=O(\ve^3),
\qquad \int_{\hat\DD} \tilde{h}_{3,\ve} Z_N d\xi=O(\ve^2 \rho),
\qquad \int_{\hat\DD} \tilde{h}_{3,\ve} Z_0 d\xi=O(\ve^3).
\end{equation}
In \equ{w2epsilon3}, $\Xi_3(d^0_{j,\ve})$ is given by
\begin{align*}
\Xi_3(d^0_{j,\ve}) =& \left\{-\mu_{0,\ve}\partial_jU\triangle_K d^0_{j,\ve}+\gamma(1+\gamma) \nabla_K\mu_{0,\ve} \nabla_Kd^0_{j,\ve} \partial_jU+2\nabla_K\mu_{0,\ve}\nabla_Kd^0_{j,\ve}\partial_{jl}^2U\xi_l\right.\nonumber\\[3mm]
 &  \left. -2\mu_{0,\ve}\tilde{g}^{ab}\frac{1}{\rho}\partial_{\bar{a}j}^2U\partial_bd^0_{j,\ve}-\frac{1}{3}\mu_\ve R_{mijl}(\xi_md^0_{l,\ve}+\xi_ld^0_{m,\ve})\partial_{ij}^2U
+\mu_{0,\ve}\mathfrak{D}_{Nl}^{ij}\xi_Nd^0_{l,\ve}\partial_{ij}^2U\right.\nonumber\\[3mm]
 & \left.+\mu_{0,\ve}[\frac{2}{3}R_{mllj}+\tilde{g}^{ab}R_{jabm}-\Gamma^c_{am}\Gamma^{a}_{cj}]d^0_{m,\ve}\partial_jv
-2\mu_{0,\ve}\xi_N(H_{aj}+\tilde{g}^{ac}H_{cj})\partial_ad^0_{l,\ve}\partial_{jl}^2U\right\}.
\end{align*}
In \equ{rivoli43} the term $ \mathcal{E}_{3,\e}$ can be described as
the sum of functions of the form (\ref{youjieas}).

Finally the term $Q_\ve (w_{3,\e} )$ is a sum of quadratic terms in
$w_{2,\e}$ like
$$
(\mu_{0,\ve}+\mu_{1,\ve}+\mu_{2,\ve})^{\frac{N-2}{2}\ve}\left[ (U-\bar{U} + w_{1,\e} + w_{2,\e}+ w_{3,\e}+\ve e_\ve Z_0 )^{p-\ve}- (U-\bar{U} + w_{1,\e}+ w_{2,\e}+\ve e_\ve Z_0  )^{p-\ve}\right.
$$
$$
\left.-(p-\ve) (U-\bar{U} +
w_{1,\e}+ w_{2,\e}+\ve e_\ve Z_0  )^{p-1-\ve}  w_{3,\e}\right]
$$
and linear terms in $w_{3,\e}$ multiplied by a term of order
$\ve$, like
$$
(p-\ve) \left( (U-\bar{U} + w_{1,\ve}+ w_{2,\e}+\ve e_\ve Z_0  )^{p-1-\ve} - U^{p-1-\ve} \right) w_{3,\ve}.
$$

\medskip
We now proceed with the choice of $\mu_{2,\ve}, d_{N,\ve}^2, e_{2,\e}$ and $d^0_{l,\ve}$, $l=1, \ldots , N-1$.

\noindent  {\it Projection onto $Z_{N+1}, Z_{N}, Z_{0}$ and choice of
$\mu_{2,\ve}, d_{N,\ve}^2, e_{2,\e}$}. \ \  Arguing as in the last step of the iteration we can prove that the three orthogonality conditions
$
 \int_{ \DD} h_{3,\ve}Z_{l}=0,\quad l=0, N, N+1.
$
 are guaranteed choosing the parameters $\mu_{2,\ve}, d_{N,\ve}^2, e_{2,\e}$, to be  solutions of the following system
$$
M\left(
   \begin{array}{c}
     \mu_{2,\ve} \\[3mm]
     d_{N,\ve}^2 \\
   \end{array}
 \right)=
 \ve^2 \left(
   \begin{array}{c}
     \tilde{\Re}_{13}(\mu_{0,\ve},\mu_{1,\ve};d_{N,\ve}^0,d_{N,\ve}^1;e_{0,\ve},e_{1,\ve}) \\[3mm]
     \Re_{23}(\mu_{0,\ve},\mu_{1,\ve};d_{N,\ve}^0,d_{N,\ve}^1;e_{0,\ve},e_{1,\ve}) \\
   \end{array}
 \right)
$$
and
\begin{eqnarray}\label{systemCtghj}
&&(N-2)A_4\frac{\mu_{0,\ve}^{N-2}}{(d_{N,\ve}^0)^{N-2}}\left(\frac{\mu_{2,\ve}}{\mu_{0,\ve}}-\frac{d_{N,\ve}^2}{d_{N,\ve}^0}\right)
+A_5-A_7\log(\mu_{2,\ve})-
\lambda_1 e_{2,\ve}
\nonumber\\
& -& 2H_{jj}  d_{N,\ve}^2\int_{\R^N}\partial_{jj}^2 UZ_0d\xi=\ve^2 \Re_{33}(\mu_{0,\ve},\mu_{1,\ve};d_{N,\ve}^0,d_{N,\ve}^1;e_{0,\ve},e_{1,\ve}). \nonumber
\end{eqnarray}
where the matrix $M$ was defined in (\ref{systematrix}). Arguing as in the previous step, we can get the existence and smoothness of   $\mu_{2,\ve}, d_{N,\ve}^2$,   $e_{2,\ve}$, solutions of the above system. Moreover, we have the validity of the following bounds on such parameters
\begin{equation*}
\label{emu1e3} \| \mu_{2,\e}\|_{L^\infty (K)} +\| \partial_a
\mu_{2,\e}\|_{L^\infty (K)} +\|\partial^2_a \mu_{2,\e}\|_{L^\infty
(K)} \leq C \ve^2,
\end{equation*}
\begin{equation*}
\label{emu1e3} \| d^2_{N,\e}\|_{L^\infty (K)} +\| \partial_a
d^2_{N,\e}\|_{L^\infty (K)} +\|\partial^2_a d^2_{N,\e}\|_{L^\infty
(K)} \leq C \ve^2,
\end{equation*}
and
\begin{equation*}
\label{emu1e3} \| e_{2,\e}\|_{L^\infty (K)} +\| \partial_a
e_{2,\e}\|_{L^\infty (K)} +\|\partial^2_a e_{2,\e}\|_{L^\infty
(K)} \leq C \ve^2.
\end{equation*}

\medskip
\noindent  {\it Projection onto $Z_l$ and choice of
$d^0_{l,\ve}$:}\ \  Multiplying $h_{3,\ve}$ by $Z_l = \pa_l U$,
integrating over $ \DD$ and using the fact $U$ is even in the
variable $\ov\xi$, one obtains
\begin{align}\label{projZl}
 \int_{\hat \DD}h_{3,\ve}\,Z_l  =&\e^2\rho\left\{ - \,  \mu_{0,\ve} \,\Delta_K d^0_{j,\ve}\,\int_{ \hat\DD} \pa_jU \pa_lU+\ve \int_{\hat \DD}\mathfrak{G}_{2,\e}\,\pa_l U \right.\nonumber\\[2mm]
    &-  {1\over 3}\,\mu_{0,\ve}\, R_{mijs}\int_{ \hat\DD}(\xi_md^0_{s,\ve}+\xi_sd^0_{m,\ve})\,\del^2_{ij}U  \pa_lU\\[2mm]
  & +\left.  \mu_{0,\ve}\,\left[\frac23\, R_{mssj}\,d^0_{m,\ve}+  \bigg( \tilde g_\e^{ab}\,R_{maaj}- \G_a^c(E_m) \G_c^a(E_j)  \bigg)
\,d^0_{m,\ve}\right]\, \int_{ \hat\DD}\pa_j U\,\pa_lU\right\}+O(\ve^2\rho).\nonumber
\end{align}
First of all, observe that  by oddness in $\ov\xi$ we have that
$$
\int_{ \hat\DD} \pa_jU\pa_lU=\d_{lj}\,\left(\int_{\R^N } |\pa_lU|^2+O(\e^{N-2})\right)=
\d_{lj}\, C_0+O(\e^{N-2})
$$
with
$
C_0:= \int_{\R^N } |\pa_lw_0|^2.
$
On the other hand the integral $\int_{\hat\DD} \xi_m
\,\del^2_{ij}U  \pa_lU$ is non-zero only if, either $i = j$ and
$m = l$, or $i = l$ and $j = m$, or $i = m$ and $j = l$. In the
latter case we have $R_{mijs}=0$ (by the antisymmetry of the
curvature tensor in the first two indices). Therefore, the first
term of the second line of the above formula becomes simply
\begin{eqnarray*}
&&  \sum_{ijms}R_{mijs}\int_{\hat\DD}\xi_m  d^0_{s,\ve}\,\del^2_{ij}U  \pa_lU\\
 &=&
  \sum_{is} R_{liis} d^0_{s,\ve} \int_{\hat\DD} \xi_l \partial_{l}
U\partial^2_{ii} Ud \xi+  \sum_{js} R_{jljs} d^0_{s,\ve}
  \int_{\hat \DD} \xi_j \partial_{l}
 U\partial^2_{lj} U d \xi  \\
& =&
  \sum_{is} R_{liis} d^0_{s,\ve} \int_{\R^N } \xi_l \partial_{l}
U \partial^2_{ii} U d \xi   +  \sum_{js} R_{jljs} d^0_{s,\ve}
  \int_{\R^N } \xi_j \partial_{l}
 U \partial^2_{lj} U d \xi+O(\e^{N-2}).
\end{eqnarray*}
Observe that, integrating by parts, when $l \neq i$ (otherwise
$R_{liis}=0$) there holds
\begin{eqnarray*}
  \int_{\hat\DD} \xi_l \partial_{l} U \partial^2_{ii}
Ud \xi
  = \int_{\R^N} \xi_l \partial_{l} U \partial^2_{ii}
U d \xi +O(\e^{N-2}) = - \int_{\R^N} \xi_l \partial_{i}
U \partial^2_{li} U d \xi+O(\e^{N-2}).
\end{eqnarray*}
Hence, still by the antisymmetry of the curvature tensor we obtain that
$$
 \sum_{ijms}R_{mijs}\int_{\hat\DD}\xi_m  d^0_{s,\ve}\,\del^2_{ij}U  \pa_lU=-2  \sum_{is}R_{liis}d_{s,\e}^0\bigg( \int_{\R^N} \xi_l \partial_{i}
U \partial^2_{li} U d \xi+O(\e^{N-2})\bigg).
$$
Then the second line in Formula \eqref{projZl} becomes (permuting the indices $s$ and $m$ in the above
argument)
\begin{eqnarray*}
&&-   {1\over 3}\,\mu_{0,\ve}\, \sum_{ijms}R_{mijs}\int_{\hat\DD}(\xi_m  d^0_{s,\ve}+\xi_s  d^0_{m,\ve})\,\del^2_{ij}U  \pa_lU\\
&=&  {4\over 3}\,\mu_{0,\ve}\,  \sum_{is}R_{liis}d_{s,\e}^0\bigg( \int_{\R^N} \xi_l \partial_{i}
U \partial^2_{li} U d \xi+O(\e^{N-2})\bigg)= -\,{2\over 3}\,\mu_{0,\ve}\,  \sum_{is}R_{liis}d_{s,\e}^0\bigg( C_0+O(\e^{N-2})\bigg)
\end{eqnarray*}
Collecting the above computations, we conclude that
\begin{eqnarray*}
-{1\over 3}\,\mu_{0,\ve}\,R_{mijl}\int_{\hat\DD} (\xi_m  d^0_{l,\ve}
+\xi_l  d^0_{m,\ve} )\,\del^2_{ij}U  \pa_lU+
 \frac23\,\mu_{0,\ve}\, R_{mssj}\,d^0_{m,\ve}\int_{\hat\DD}\pa_jU\,\pa_lU=O(\e^{N-2}).
\end{eqnarray*}
Hence formula \eqref{projZl} becomes simply
\begin{align*}
[\mu_{0,\ve}\e^2\rho]^{-1}\int_{\hat\DD}h_{3,\e}\,\pa_l U =& -     C_0\,\Delta_K\,d^0_{l,\ve}+    C_0\,\bigg( \tilde g_\e^{ab}\,R_{maal}- \G_a^c(E_m) \G_c^a(E_l)  + O(\e^{N-2}) \bigg)
\,d^0_{m,\ve}\\
 +&  \int_{\hat\DD}\mathfrak{G}_{2,\e}\,\pa_l w_0.
\end{align*}
We thus obtain that $h_{3,\e}$, the
right-hand side of \eqref{edith2}, is $L^2$-orthogonal to $Z_l$
($l=1,\cdots,N-1$) if and only if $d^0_{l,\ve}$ satisfies an
equation of the form
\begin{equation}\label{phi1def}
  \Delta_K \,d^0_{l,\ve}-\bigg( \tilde g_\e^{ab}\,R_{maal}- \G_a^c(E_m) \G_c^a(E_l) + O(\e^{N-2} ) \bigg)
\,d^0_{m,\ve} = G_{2,\e}(\rho z),
\end{equation}
for some smooth function $G_{2,\e}$, whose $L^\infty$ norm on $K$ is bounded by a fixed constant, as $\e \to 0$. Observe that
the operator acting on $d^0_{l,\ve}$ in the left hand side is
nothing but the Jacobi operator of the sub manifold $K$. By our assumption,  $K$ is non degenerate and hence this operator is invertible.  This implies the solvability of the
above equation in $d^0_{l,\ve}$.
Furthermore, equation \equ{phi1def} defines $d^0_{l,\ve}$ as a
smooth function on $K$, with
\begin{equation}
\label{ePhi1e} \| d^0_{l,\ve}\|_{L^\infty (K)} +\| \partial_a
d^0_{l,\ve}\|_{L^\infty (K)} +\|\partial^2_a d^0_{l,\ve}\|_{L^\infty
(K)} \leq C \quad l=1,\ldots,N-1.
\end{equation}

\medskip
\noindent
Given the choice of the parameters $\mu_{2,\e}, d^2_{N,\ve}$, $e_{2,\ve}$ and $d^0_{l,\ve}(l=1,\ldots,N-1)$, the linear operator defined in \eqref{rivoli43} by \eqref{rivoli3}, which  depends on $\mu_\e$, $d_{N,\e}$ and $e_\e$, satisfies the assumptions of Proposition \ref{linear}. Furthermore, we have the existence of $w_{3,\e}$ solution to
\begin{equation}\label{edith23}
   \mathcal{L}_{\ve}w_{3,\e} =   h_{3,\e}  \quad\hbox{ in } \DD,\qquad
     w_{3,\e}  = 0, \ \  \hbox{ on } \partial\DD.
\end{equation}
Moreover, for a given $\sigma \in (0,1)$ we have
$
\|   h_{3,\e} \|_{\e, 4, \sigma} \leq C \ve^3.
$
Proposition \ref{linear} thus gives then that
\begin{equation}
\label{ew2} \| D^2_\xi w_{3,\e} \|_{\ve , 4, \sigma } + \| D_\xi
w_{3,\e} \|_{\ve , 3 , \sigma } +\|w_{3,\e} \|_{\ve, 2 ,
\sigma}\le C \ve^{3}
\end{equation}
and that there exists a positive constant $\beta$ (depending only on
$\Omega, K$ and $n$) such that for any integer $\ell$ there holds
\begin{equation}\label{eq:estw2}
    \|\nabla^{(\ell)}_{y} w_{3,\e}(z,\cdot)\|_{\ve,N-3,\sigma} \leq \beta C_\ell \,
    \ve^{3} \qquad \quad y =\rho z \in K.
\end{equation}
where $C_\ell$ depends only on $\ell$, $p$, $K$ and $\Omega$.
Moroever, we have that
$$
\| - {\mathcal A}_{\mu_\ve , d_\e } v_{3,\e}   - \mu_\ve^{\frac{N-2}{2}\ve}\,v_{3,\e}^{p- \ve} \|_{\e, 4 , \sigma} \leq C \e^4.
$$

\medskip
\subsubsection*{\bf Expansion at an arbitrary order}

\noindent We take now an arbitrary integer $I$, we let
\begin{equation}\label{eqmu}
 \mu_\ve:= \mu_{0,\ve}+\mu_{1,\e}+\cdots +\mu_{I-1,\e} + \mu_{I, \ve},
\end{equation}
$$
d_{l,\ve}=d^0_{l,\ve}+\ldots+d_{l,\ve}^{I-2}.\quad l=1,\ldots,N-1;\qquad
d_{N,\ve}=d^0_{N,\ve}+\ldots+d_{N,\ve}^I
$$
and
$$
e_\ve=e_{0,\ve}+e_{1,\ve}+\cdots+e_{I,\ve}
$$
and we define
\begin{equation}\label{eqWWW}
v_{I+1, \ve} = U (\xi ) -\bar{U}(\xi)+ w_{1, \ve} (z, \xi) +
\ldots + w_{I, \ve} (z, \xi ) + w_{I+1, \ve} (z, \xi )+\ve e_\ve \chi_\ve Z_0
\end{equation}
where $\mu_{0,\ve},  \mu_{1,\e} , \cdots ,  \mu_{I-1,\e}$, $d^1_{l,\e},
\cdots , d^{I-3}_{l,\e}$, $d^0_{N,\ve},\ldots,d^{I-1}_{N,\ve}$, $e_{0,\ve},e_{1,\ve},\ldots,e_{I-1,\ve}$ and $w_{1, \ve}, .. ,w_{I, \ve}$ have
already been constructed following an iterative scheme, as described
in the previous steps of the construction.

In particular one has, for any $i=1, \ldots , I-1$,
\begin{eqnarray*}
&& \| \mu_{i,\e}\|_{L^\infty (K)} +\| \partial_a
\mu_{i,\e}\|_{L^\infty (K)} +\|\partial^2_a \mu_{i,\e}\|_{L^\infty
(K)} \leq C \ve^{i},\\[3mm]
&&  \| d^i_{N,\e}\|_{L^\infty (K)} +\| \partial_a
d^i_{N,\e}\|_{L^\infty (K)} +\|\partial^2_a d^i_{N,\e}\|_{L^\infty
(K)} \leq C \ve^{i-1},\\[3mm]
&& \| d^i_{l,\e}\|_{L^\infty (K)} +\| \partial_a
d^i_{l,\e}\|_{L^\infty (K)} +\|\partial^2_a d^i_{l,\e}\|_{L^\infty
(K)} \leq C \ve^{{i-1}},\quad l=1,\ldots,N-1,\\[3mm]
&& \| e_{i,\e}\|_{L^\infty (K)} +\| \partial_a
e_{i,\e}\|_{L^\infty (K)} +\|\partial^2_a e_{i,\e}\|_{L^\infty
(K)} \leq C \ve^{{i-1}},
\end{eqnarray*}
and moreover for any  $i=0, \ldots , I-1$ we have that
\begin{equation*}
\label{ewi} \| D^2_\xi w_{i+1,\e} \|_{\ve , 4 , \sigma } + \|
D_\xi w_{i+1,\e} \|_{\ve , 3 , \sigma } +\|w_{i+1,\e} \|_{\ve, 2
, \sigma}\le C \ve^{i+1}
\end{equation*}
and for any integer $\ell$
\begin{equation*}\label{eq:estwi}
    \|\nabla^{(\ell)}_{z} w_{i+1,\e}(z,\cdot)\|_{\ve,2,\sigma} \leq \beta C_l \ve^{i+1},
     \qquad \quad z \in K_\rho.
\end{equation*}

The new components $(\mu_{I, \ve} , d^{I-2}_{1,\ve},\ldots, d^{I-2}_{N-1,\ve},d^I_{N, \ve} , e_{I, \ve} )$
will be found reasoning as before. Computing $S
(v_{I+1 , \ve })$ (see \equ{Sep}) we get
\begin{eqnarray}\label{rivoli43i}
\mathcal{S}_\ve( v_{I+1,\e}  ) = \mathcal{L}_{\e} w_{I+1, \e}-  h_{I+1, \e}+   \ve^{I+2}\mathcal{E}_{I+1,\e}    +Q_\ve (w_{I+1,\e} )
\end{eqnarray}
where $\mathcal{L}_{\e}$ is defined in \equ{rivoli3},  and  the function $h_{3,\e}$ is given by
\begin{eqnarray}\label{w2epsilon3i}
h_{I+1,\ve} &= &  -2\ve d_{N,\ve}^I H_{ij}\partial_{ij}^2U  +\rho \, \mu_{I,\e}\ \left\{-2\xi_NH_{ij}\partial_{ij}^2U+H_{\alpha\alpha}\partial_NU\right\} -\lambda_1\,\e\,e_{I,\e}\, Z_0\nonumber\\
&& -\ve\frac{N-2}{2}\frac{\mu_{I,\e}}{\mu_{0,\e}}U^p+ 2pU^{p-1}\partial_N U\big( \bar{\xi},\xi_N+2\frac{ \ve d^0_{N,\ve}}{\rho \mu_{0,\ve}} \big)\,\frac{ \ve  d^0_{N,\ve}   }{\rho \mu_{0,\ve} }\left[\frac{d^I_{N, \e}}{d^0_{N, \e}}-\frac{\mu_{I,\e}}{\mu_{0,\e}}\right]\\
&&+\ve^I\,\rho\,\Xi_{I+1}(d^{I-2}_{j,\ve})+ \tilde{h}_{I+1,\ve}\nonumber
\end{eqnarray}
where $\tilde{h}_{I+1,\ve}$ is a smooth function on its variable which depends only on the parameters $\mu_{j,\ve}$, $d_{N,\ve}^j$, $d^{j}_{\ell,\ve}$, $e_{j,\ve}$ which have been constructed in the previous steps.
with
\begin{equation}\label{tildef2epro13i}
\int_{\hat\DD} \tilde{h}_{I+1,\ve} Z_jd\xi=O(\ve^I\rho),\quad j=1,\ldots,N-1,
\end{equation}
and
\begin{equation}\label{tildef2epro23i}
\int_{\hat\DD} \tilde{h}_{I+1,\ve} Z_{N+1}d\xi=O(\ve^{I+1}),
\qquad \int_{\hat\DD} \tilde{h}_{I+1,\ve} Z_N d\xi=O(\ve^I \rho),
\qquad \int_{\hat\DD} \tilde{h}_{I+1,\ve} Z_0 d\xi=O(\ve^I \rho).
\end{equation}

In \equ{w2epsilon3}, $\Xi_{I+1}(d^{I-2}_{j,\ve})$ is given by
\begin{align*}
 \Xi_{I+1}(d^{I-2}_{j,\ve})
 =&  -\mu_{0,\ve}\partial_jU\triangle_K d^{I-2}_{j,\ve}+\gamma(1+\gamma) \nabla_K\mu_{0,\ve} \nabla_Kd^{I-2}_{j,\ve} \partial_jU+2\nabla_K\mu_{0,\ve}\nabla_Kd^{I-2}_{j,\ve}\partial_{jl}^2U\xi_l \nonumber\\[3mm]
 &   -2\mu_{0,\ve}\tilde{g}^{ab}\frac{1}{\rho}\partial_{\bar{a}j}^2U\partial_bd^j_{I-1,\ve}-\frac{2}{3}\mu_\ve R_{islj}\xi_sd^{I-2}_{l,\ve}\partial_{ij}^2U
+\mu_{0,\ve}\mathfrak{D}_{Nl}^{ij}\xi_Nd^{I-2}_{l,\ve}\partial_{ij}^2U \nonumber\\[3mm]
 &  +\mu_{0,\ve}[\frac{2}{3}R_{mllj}+\tilde{g}^{ab}R_{jabm}-\Gamma^c_{am}\Gamma^{a}_{cj}]d^{I-2}_{m,\ve}\partial_jv
-2\mu_{0,\ve}\xi_N(H_{aj}+\tilde{g}^{ac}H_{cj})\partial_ad^{I-2}_{l,\ve}\partial_{jl}^2U.
\end{align*}
In \equ{rivoli43i} the term $ \mathcal{E}_{I+1,\e}$ can be described as
the sum of functions of the form (\ref{youjieas}).
Finally the term $Q_\ve (w_{I+1,\e} )$ in \equ{w2epsilon3i} is a sum of quadratic terms
in $w_{I+1,\e}$ like
\begin{eqnarray*}
(\mu_{0,\ve}+\mu_{1,\e}+\cdots +\mu_{I-1,\e} + \mu_{I, \ve})^{  \frac{N-2}{2}\e}\left[v_{I+1, \ve}^{p- \e} - v_{I, \ve}^{p- \e}  -(p- \e) v_{I, \ve}^{p-1- \e}
w_{I+1,\e}\right]
\end{eqnarray*}
and linear terms in $w_{I+1,\e}$ multiplied by a term of order
$\ve^2$, like
$$
(p-\ve) \left( (U-\bar{U} + w_{1,\ve})^{p-1-\ve} - (U-\bar{U})^{p-1-\ve } \right)
w_{I+1,\ve}.
$$
Arguing as in the previous step, it is possible to prove the existence of parameters $\mu_{I,\e}$ and the normal section $d^{I-2}_{1,\e},\ldots,d^{I-2}_{N-1,\e},d^I_{N,\e}$ and $e_{I,\ve}$
in such a way that  $h_{I+1,\ve}$ is $L^2$-orthogonal to $Z_j$, $j=0,
1,\cdots,N+1$. Furthermore,
\begin{equation*}
\label{emu1e3I} \| \mu_{I,\e}\|_{L^\infty (K)} +\| \partial_a
\mu_{I,\e}\|_{L^\infty (K)} +\|\partial^2_a \mu_{I,\e}\|_{L^\infty
(K)} \leq C \ve^I,
\end{equation*}
\begin{equation*}
\label{emu1e3I} \| d^I_{N,\e}\|_{L^\infty (K)} +\| \partial_a
d^I_{N,\e}\|_{L^\infty (K)} +\|\partial^2_a d^I_{N,\e}\|_{L^\infty
(K)} \leq C \ve^I,
\end{equation*}
\begin{equation*}
\label{emu1e3I} \| e_{I,\e}\|_{L^\infty (K)} +\| \partial_a
e_{I,\e}\|_{L^\infty (K)} +\|\partial^2_a e_{I,\e}\|_{L^\infty
(K)} \leq C \ve^I.
\end{equation*}
and
\begin{equation}
\label{ePhiIe} \| d^{I-2}_{l,\ve}\|_{L^\infty (K)} +\| \partial_a
d^{I-2}_{l,\ve}\|_{L^\infty (K)} +\|\partial^2_a d^{I-2}_{l,\ve}\|_{L^\infty
(K)} \leq C \ve^{I-2}.
\end{equation}
We are now in a position to apply Proposition \ref{linear} to get $w_{I+1, \e}$ solution to
\begin{equation}\label{eq:eqwIi}
    \mathcal{L}_{\ve}w_{I+1,\e} =  h_{I+1,\e}   \quad\hbox{ in } \DD\qquad
     w_{I+1,\e}  = 0, \ \,   \hbox{ on } \partial\DD.
\end{equation}
where $\mathcal{L}_{\ve}$ is defined in \equ{rivoli3}.
Furthermore, we have that
\begin{equation}
\label{ewI} \| D^2_\xi w_{I+1,\e} \|_{\ve , 4 , \sigma } + \|
D_\xi w_{I+1,\e} \|_{\ve , 3 , \sigma } +\|w_{I+1,\e} \|_{\ve, 2
, \sigma}\le C \ve^{I+1}
\end{equation}
and that there exists a positive constant $\beta$ (depending only on
$\Omega, K$ and $N$) such that for any integer $\ell$ there holds
\begin{equation}\label{eq:estw2}
    \|\nabla^{(\ell)}_{y} w_{I+1,\e}(z,\cdot)\|_{\ve,2,\sigma} \leq \beta C_l
    \ve^{I+1} \qquad \quad   y = \rho z \in K.
\end{equation}
With this choice of $(\mu_{I, \ve} , d^{I-2}_{1,\ve},\ldots, d^{I-2}_{N-1,\ve},d^I_{N, \ve} , e_{I, \ve} )$ and $w_{I+1 , \e}$ we obtain that
$$
\| - {\mathcal A}_{\mu_\ve , d_\ve } v_{I+1,\e}   - \mu_\ve^{\frac{N-2}{2}\ve}\,v_{I+1,\e}^{p- \ve} \|_{\e , 4 , \sigma} \leq C \e^{I+2}.
$$

This concludes our construction and have the validity of Proposition \ref{Construction}.

\section{The existence result: Proof of the Theorem \ref{teo1}}\label{exct}
\setcounter{equation}{0}

Let us recall that if $u$ is a solution to problem (\ref{eq:pe}), and defining $\tilde{u}$ by
$$
u(x) =(1+\alpha_\e)\rho^{-{N-2 \over 2}}
\tilde{u}( \rho^{-1}x ),
$$
then $\tilde{u}$ satisfies the following equation
\begin{equation} \label{changea}
-\Delta \tilde{u} =  \tilde{u}^{{N+2\over N-2} -\e}, \quad \tilde{u}>0\quad
\mbox{ in }
\Omega_\rho ; \qquad
 \tilde{u}=0 \quad \mbox{ on }\ \ \partial\Omega_\rho,
\end{equation}
where $\Omega_\rho=\frac{\Omega}{\rho}$.
\subsection{Global approximate solution}
Let $I$ be an integer. We have constructed an approximate solution $v_{I+1,\ve}$ in Section \ref{s:aprsol}, such that
$$
\mathcal{S}_\ve(v_{I+1,\e})=-{\mathcal A}_{\mu_\ve , d_\e } v_{I+1,\e} -\mu_{\ve}^{\frac{N-2}{2}\ve}v_{I+1,\e}^{\frac{N+2}{N-2}-\ve}=\mathcal{O}(\ve^{I+2}) \quad {\mbox {in}} \quad K_\rho\times \hat\DD,
$$
where $\mu_\e (y)  $ and $d_\e (y)$ are functions defined on $K$, whose existence and
properties are established in Proposition \ref{Construction}.
We
define locally around $K_\rho $  the function
\begin{eqnarray} \label{Vdef}
\tilde{U}_\e (z, x) &:=&  \, \mu_{\e}^{-{N-2 \over 2}} (\rho z) \, v_{I+1 , \e}
\left( z, \,  \frac{\bar{x}- \e^2\rho^{-1}\bar{d}_\e (\rho z)}{ \mu_\e (\rho z)}\, , \frac{x_N- \e \rho^{-1}d_{N,\e} (\rho z)}{ \mu_\e (\rho z)} \right)   \nonumber\\
&&\times\chi_\e (|(\bar{x}- \e^2\rho^{-1}\bar{d}_\e , x_N- \e \rho^{-1}d_{N,\e})|)
\end{eqnarray}
where $z \in K_\rho$. Here $\chi_\e$
is a smooth cut-off function with
\begin{equation}\label{magaly}
 \chi_\e (r) =1,\ \hbox{for} \ r \in [0,2 \e^{-\gamma} ],\ \
   \chi_\e (r) =  0,\ \hbox{for}\ r \in [3 \e^{-\gamma}, 4\e^{-\gamma} ],\ \
 \mbox{and}\ \ |\chi_\e^{(l)} (r) | \leq C_l \e^{l \gamma}, \ \ \forall l\geq 1,
\end{equation}
for some $\gamma \in ({1\over 2} , 1)$ to be fixed later.

We will use the notation
\begin{equation} \label{defTTa}
\tilde{u}=\tilde{{\mathcal T}}_{\mu_\e , d_\e} (\tilde{v})
\end{equation}
if and only if $\tilde{u}$ and $\tilde{v}$ satisfy
\begin{equation*}
 \tilde{u}=\mu_{\e}^{-{N-2 \over 2}} (\rho z) \, \tilde{v}
\left( z, \,  \frac{\bar{x}- \e^2\rho^{-1}\bar{d}_\e (\rho z)}{ \mu_\e (\rho z)}\, , \frac{x_N- \e\rho^{-1} d_{N,\e} (\rho z)}{ \mu_\e (\rho z)} \right) .
\end{equation*}
The function $\tilde{U}_\e$ is globally defined in
$\Omega_\rho$. We will look for  a solution to (\ref{changea})
of the form
$$
\tilde{u}_\e = \tilde{U}_\e +\phi,
$$
where $\phi$ is a lower term. Thus $\phi$ satisfies the following problem
\begin{equation}\label{nonlinearproblem}
  L_\e (\phi):=  -\Delta  \phi  -  (p-\e) \tilde{U}_\e^{p -1-\e} \phi =S_\e (\tilde{U}_\e) + N_\e (\phi)  \quad  \text{ in } \Omega_\rho,\quad \phi=0\quad \mbox{on}\ \ \partial\Omega_\rho,
\end{equation}
where
\begin{equation}
\label{eomegaeps}
S_\e (\tilde{U}_\e )= \Delta_{g^\rho} \tilde{U}_\e   +
\tilde{U}_\e^p,
\end{equation}
and
\begin{equation}
\label{Nomegaeps}
N_\e (\phi) =  (\tilde{U}_\e + \phi )^{p-\e} - \tilde{U}_\e^p -  (p-\e)
\tilde{U}_\e^{p -1-\e} \phi,
\end{equation}
where $g^\rho(y,x)=g(\rho y,\rho x)$.

To solve the Non-Linear Problem
(\ref{nonlinearproblem}) we use a fixed point argument based on the
contraction Mapping Principle. First we establish some
invertibility properties of the linear problem
$$
L_\e (\phi) = f \quad {\mbox {in}} \quad \Omega_\rho ,\quad \phi=0\quad {\mbox {on}} \quad \partial\Omega_\rho
$$
with $f\in L^2 (\Omega_\rho )$. This is the purpose of the next result.

\begin{proposition}\label{glinear}
There exist a sequence $\e_l \to 0$ and a positive
constant $C >0$, such that, for any $f \in L^2 (\Omega_{\rho_l} )$,
there exists a solution $\phi \in H^1_0 (\Omega_{\rho_l} )$ to the
equation
$$
L_{\e_l }\phi = f \quad \hbox{in } \Omega_{\rho_l},\quad \phi=0\ \ \mbox{on}\ \partial\Omega_{\rho_l},
$$
with $\rho_l=\e_l^{\frac{N-1}{N-2}}$.
Furthermore,
\begin{equation}
\label{gigio}
\| \phi \|_{H^1_0  (\Omega_{\rho_l})} \leq C\, \rho_l^{-\max \{
2, k\}} \| f \|_{L^2 (\Omega_{\rho_l})}.
\end{equation}
\end{proposition}

The proof of this proposition will be given in Section \ref{linearas}. We are now in position to prove our main Theorem \ref{teo1}.

\smallskip

\subsection{ Proof of the main Theorem \ref{teo1}}
By Proposition \ref{glinear}, $\phi \in H^1_0 (\Omega_{\rho} )$ is a
solution to (\ref{nonlinearproblem}) if and only if
$$
\phi = L_\e^{-1} \left( S_\e (\tilde{U}_\e ) + N_\e (\phi )
\right).
$$
Notice that
\begin{equation}
\label{nonno1} \| N_\e (\phi ) \|_{L^2 (\Omega_{\rho} )} \leq C
\begin{cases}
    \| \phi \|_{H^1_0 (\Omega_{\rho})}^p & \text{ for } p\leq 2, \\
    \| \phi \|_{H^1_0 (\Omega_{\rho} )}^2 & \text{ for } p>2
  \end{cases} \quad\qquad  \| \phi \|_{H^1_0 (\Omega_{\rho} )} \leq 1
\end{equation}
and
\begin{eqnarray}
\label{nonno2}
 \| N_\e (\phi_1 ) - N_\e (\phi_2 ) \|_{L^2
(\Omega_{\rho}
)}
  \leq  C \,\begin{cases}
   \left(  \| \phi_1  \|_{H^1_0 (\Omega_{\rho})}^{p-1} + \| \phi_2  \|_{H^1_{g^\e} (\Omega_{\rho})}^{p-1} \right)
   \| \phi_1 -\phi_2 \|_{H^1_0  (\Omega_{\rho} )} & \text{ for } p\leq 2, \\[3mm]
    \left(  \| \phi_1  \|_{H^1_0 (\Omega_{\rho} )} + \| \phi_2  \|_{H^1_0  (\Omega_{\rho})} \right) \| \phi_1 -\phi_2 \|_{H^1_0  (\Omega_{\rho} )}  & \text{ for } p>2
  \end{cases},
\end{eqnarray}
for any $\phi_1$, $\phi_2$ in $H^1_0  (\Omega_{\rho} )$ with $\| \phi_1
\|_{H^1_0  (\Omega_{\rho} )}$, $ \| \phi_2 \|_{H^1_0 (\Omega_{\rho})} \leq 1.$

Defining $T_\e :H^1_0  (\Omega_{\rho} ) \to H^1_0 (\Omega_{\rho} )$ as
$$
T_\e (\phi ) = L_\e^{-1} \left( S_\e (\tilde{U}_\e ) + N_\e (\phi
) \right)
$$
we will show that $T_\e$ is a contraction in some small ball in $H^1_0 (\Omega_{\rho} )
$. A direct consequence of (\ref{bf4}), we have
$
\|S_\e (\tilde{U}_\e )\|_{L^2(\Omega_{\rho})}\leq C\e^{I+1}.
$
Using this inequality and by
(\ref{nonno1}),
(\ref{nonno2}) and (\ref{gigio}), we obtain
$$
\| T_\e (\phi ) \|_{H^1  (\Omega_{\rho} )} \leq C \rho^{- \max \{ 2 , k \}}
\begin{cases}
   \left( \e^{I+1 }+  \| \phi  \|_{H^1_0 (\Omega_{\rho} )}^p \right) & \text{ for } p\leq 2, \\
    \left( \e^{I+1} + \| \phi  \|_{H^1_0  (\Omega_{\rho} )}^2 \right)   & \text{ for } p>2.
  \end{cases}
$$
Now we choose integers $d$ and $I$ so that
$$
d> \begin{cases}
   \frac{N-1}{N-2}{\max \{ 2 ,k \} \over p-1}& \text{ for } p\leq 2, \\[3mm]
    \frac{N-1}{N-2}\max \{ 2 ,k \}  & \text{ for } p>2
  \end{cases} \quad I > d-1 + \frac{N-1}{N-2}\max \{ 2 , k \}.
  $$
  Thus one easily gets that $T_\e$ has a unique fixed point in set
  $${\mathcal B} = \{ \phi \in H^1_0 (\Omega_{\rho} )\, : \, \| \phi \|_{H^1_0  (\Omega_{\rho} )} \leq \e^d \},
  $$
  as a direct application of the contraction mapping Theorem. This concludes the proof.
\hfill

\section{The linear theory: proof of Proposition \ref{glinear}}\label{linearas}

\setcounter{equation}{0}

\smallskip

In this section, we will establish a solvability theory for the linear problem to prove Proposition \ref{glinear}.
We first study  the above problem in a strip close to the scaled
manifold $K_\rho$.
Let $\gamma \in ({1\over 2} , 1)$ be the number fixed before in
(\ref{magaly}) and define
\begin{equation}
\label{omegaepsilongamma}
\Omega_{\rho, \gamma} := \{ x \in \Omega_\rho
\, : \, {\mbox {dist}}  (x,K_\rho ) <2 \e^{-\gamma} \}.
\end{equation}
We are first interested in solving the following problem: given $f
\in L^2 (\Omega_{\rho, \gamma})$
\begin{equation} \label{lineare}
    -\Delta  \phi  -  (p-\e)  \tilde{U}_\e^{p -1-\e} \phi =f\quad \text{ in } \Omega_{\rho, \gamma},\quad \phi=0\quad \mbox{on}\ \partial\Omega_{\rho, \gamma}.
\end{equation}

We have the validity of the following result.
\begin{proposition}\label{teouffa}
There exist a constant $C>0$ and a sequence $\e_l = \e \to 0$ such
that, for any $f \in L^2 (\Omega_{\rho, \gamma} )$ there exists a
solution $\phi \in H^1_0 (\Omega_{\rho, \gamma}) $ to Problem (\ref{lineare}) such that
\begin{equation}
\label{uffa1}
\| \phi \|_{ H^1_0 (\Omega_{\rho, \gamma})} \leq C \rho^{- \max \{2, k\}} \| f
\|_{L^2 (\Omega_{\rho, \gamma})}.
\end{equation}
\end{proposition}

\begin{proof}
The quadratic functional of problem (\ref{lineare})
given by
\begin{equation}\label{functional1}
E(\phi ) = {1\over 2} \int_{\Omega_{\rho,\gamma}}
(|\nabla  \phi |^2  -  (p-\e) \tilde{U}_\e^{p-1-\e} \phi^2 )
\end{equation}
for functions $\phi \in H^1(\Omega_{\rho,\gamma}) $.

Let
$(y,x)\in \mathbb{R}^{k+N}$ be the local coordinates along $K_\rho$. With abuse of notation we will denote
\begin{equation}
\label{bb}
  \phi (\U(y, x))= \phi(z, x ),\quad \mbox{with}\ \ y=\rho z.
\end{equation}
Since the original variable $(z, x)\in \mathbb{R}^{k+N}$
are only local coordinates along $K_\rho$ we let the variable $(z, x)$
vary in the set $\mathcal{C}_\e$ defined by
\begin{equation}\label{ddomain}
\mathcal{C}_\e = \{ (z,x) \ / \ \rho z\in \ K,\quad    | x|  <  \e^{-\gamma} \}.
\end{equation}
We write $\mathcal{C}_\e=\frac{1}{\rho} K\times\hat {\mathcal{C}_\e}$ where
\begin{equation}\label{dddomain}
\hat {\mathcal{C}_\e} = \{ x \ / \  |x|  <  \e^{-\gamma} \}.
\end{equation}
Observe  that $\hat {\mathcal{C}_\e}$ approaches, as $\e \to 0$, the whole
space $\mathbb{R}^N$.

In these new local coordinates,  the energy density associated to
the energy $E$ in (\ref{functional1}) is given by
\begin{equation} \label{e}
\frac12 \left[|\nabla  \phi|^2  - (p-\e) \tilde{U}_\e^{p -1-\e}
\phi^2  \right] \sqrt{\det(g^\e)},
\end{equation}
where $\nabla_{g^\e}$ denotes the gradient in the new variables and
where $g^\e$ is the metric in the coordinates
$(z, x)$. Arguing as in \cite{demamu}, we have that, if $(z,x)$ vary in $\mathcal{C}_\e$,
then, the energy functional (\ref{functional1}) in the new variables
(\ref{bb}) is given by
\begin{eqnarray}\label{energydensity}
E  ( \phi) & = &\int_{K_\rho \times \hat{\mathcal{C}_\e}} \left(\frac12 (
|\nabla_x \phi|^2   - (p-\e) \tilde{U}_\e^{p -1-\e} \phi^2 )  \right)
\sqrt{\det(g^\e)} \, dz \, dx\nonumber\\
&& +\int_{K_\rho \times \hat{ \mathcal{C}_\e}}\frac12\,\Xi_{ij} (\rho z ,
x)\,\partial_{i}\phi\partial_{j}\phi\,\sqrt{\det(g^\e)} \, dz \, dx\\
 &&
 +\frac12 \int_{K_\rho\times \hat{\mathcal{C}_\e}} |\nabla_{K_\rho}\phi|^2\,\sqrt{\det(g^\e)} \, dz \, dx+
 \int_{K_\e \times \hat{\mathcal{C}_\e}} B(\phi,\phi)\,\sqrt{\det(g^\e)} \, dz \, dx, \nonumber
\end{eqnarray}
where
\begin{equation}\label{eqXiij}
\Xi_{ij} (\rho z , x) =2 \rho H_{ij} x_N -\frac{\rho^2}3\,R_{islj}x_l x_s-\rho^2 x_N^2 (H^2)_{ij},
\end{equation}
and we denoted by $B(\phi,\phi)$ a quadratic term in $\phi$ that can be
expressed in the following form
\begin{eqnarray}\label{defBB}
 B(\phi,\phi)= O \left(   \rho^3| x  |^3   \right)\partial_{i} \phi \partial_{j} \phi  +{\rho }\,|\nabla_{K_\e}\phi|^2\,
  O(\rho^2 |x|) +\partial_j \phi \partial_{\bar a}\phi \left(\mathcal{O}(\rho |x
|)\right)
\end{eqnarray}
and we used the Einstein convention over repeated indices.
Furthermore we use the notation $\partial_a =
\partial_{y_a} $ and $\partial_{\bar a} = \partial_{z_a}$.

\medskip
\noindent
Given a function $\phi \in H^1 (\Omega_{\rho,\gamma})$, we decompose it
as
\begin{equation}\label{decomp}
\phi= \left[{\delta \over \, \mu_\e} \tilde{{\mathcal T}}_{\, \mu_\e ,
\, d_\e} ( Z_{N+1}) + \sum_{j=1}^{N} {d^j \over \, \mu_\e}
\tilde{{\mathcal T}}_{\, \mu_\e , \, d_\e} ( Z_j)
 + {e \over \, \mu_\e}   \tilde{{\mathcal T}}_{\, \mu_\e , \, d_\e} (Z_0) \right] \bar \chi_\e   + \phi^\bot
\end{equation}
where the expression $\tilde{{\mathcal T}}_{\, \mu_\e , \, d_\e} (
v)$ is defined in (\ref{defTTa}), the functions $Z_{N+1}$ and $Z_j$ are
already defined in (\ref{lezetas}) and where $Z_0$ is the eigenfunction, with $\int_{\mathbb{R}^N} Z^2 =1$,
corresponding to the unique positive eigenvalue $\lambda_1$ in $L^2
(\mathbb{R}^N)$ of the problem
\begin{equation}\label{lambda0}
\Delta_{\mathbb{R}^N} \phi + p U^{p-1} \phi = \lambda_1 \phi
\quad {\mbox {in}} \quad \mathbb{R}^N.
\end{equation}
It is worth mentioning that $Z_0(\xi )$ is even and it has
exponential decay of order $O(e^{-\sqrt{\lambda_1} |\xi|} )$ at
infinity. The function $\bar \chi_\e$ is a smooth cut off function
defined by
\begin{equation}
\label{chibar} \bar \chi_\e (x) = \hat \chi_\e \left(  \left|\left({  \bar{x}-
\e^2\rho^{-1}\bar{d}_\e \over \,  \mu_\e}, {  x_N-
\e \rho^{-1}d_{N,\e} \over \,  \mu_\e}\right) \right| \right),
\end{equation}
with $\hat \chi(r) = 1$ for $r \in (0,{3\over 2} \e^{-\gamma} )$,
and $\chi(r)=0$ for $r>2\e^{-\gamma}$. Finally, in (\ref{decomp}) we
have that $\delta = \delta (\rho z)$, $d^j = d^j (\rho z)$ and $e= e(\rho
z)$ are function defined in $K$ such that $\forall z\in K_\rho$
\begin{equation}
\label{orth1}
\int_{\hat{\mathcal{C}}_\e} \phi^\bot {\tilde{{\mathcal T}}}_{\, \mu_\e , \, d_\e} (Z_{N+1}) \bar \chi_\e d x =
\int_{\hat{\mathcal{C}}_\e} \phi^\bot {\tilde{{\mathcal T}}}_{\, \mu_\e ,
\, d_\e} (Z_j) \bar \chi_\e = \int_{\hat{\mathcal{C}}_\e}
\phi^\bot {\tilde{{\mathcal T}}}_{\, \mu_\e , \, d_\e} (Z_0) \bar
\chi_\e=0.
\end{equation}
We will denote by $(H_\e^1)^\bot$ the subspace of the functions in
$H_\e^1$ that satisfy the orthogonality conditions (\ref{orth1}).

A direct computation shows that
$$
\delta (\rho z) = {\int \phi  {\tilde{{\mathcal T}}}_{\, \mu_\e ,
d_\e} (Z_{N+1}) \over \,   \mu_\e \int Z_{N+1}^2} (1+ O(\e )) + O(\e )
(\sum_j d^j (\rho z) + e (\rho z)), \quad
$$
$$
d^j (\rho z) = {\int \phi  {\tilde{{\mathcal T}}}_{\, \mu_\e , \, d_\e}
(Z_j) \over \,   \mu_\e \int Z_j^2} (1+ O(\e ))  + O(\e ) (\delta
(\rho z) + \sum_{i\not= j}  d^i (\rho z) + e (\rho z)),
$$
and
$$
e(\rho z) = {\int \phi  {\tilde{{\mathcal T}}}_{\, \mu_\e , \, d_\e} (Z_0)
\over \,   \mu_\e \int Z_0^2} (1+ O(\e))  + O(\e) (\delta (\rho z)
+\sum_j d^j (\rho z) ).
$$
Observe that, since $\phi \in H_{g^\e}^1$, one easily get that the
functions $\delta $, $d^j$ and $e$ belong to the Hilbert space
\begin{equation}\label{H1K}
{\mathcal H}^1 (K) = \{ \zeta \in {\mathcal L}^2 (K) \, : \,
\partial_a \zeta \in {\mathcal L}^2 (K),\quad a=1,\cdots,k \}.
\end{equation}

Observe that in the region we are considering the function $\tilde{U}_\ve$
is nothing but $\tilde{U}_\ve= {\tilde{{\mathcal T}}}_{\, \mu_\e , \, d_\e}
(v_{I+1 , \e})$, where $v_{I+1, \e}$ is the function whose existence
and properties are proven in Lemma \ref{Construction}. For the
argument in this part of our proof it is enough to take $I=3$, and
for simplicity of notation we will denote by $\hat w$ the function
$v_{I+1 , \e}$ with $I=3$. Referring to (\ref{bf4}) we have
\begin{equation}\label{hatw}
\hat w (z, \xi) = U (\xi)-\bar{U}(\xi) + \sum_{i=1}^4 w_{i,\e} (z,
\xi)
\end{equation}
where $U=w_N$ and $\bar{U}$ are defined in \equ{wn} and (\ref{ubar}), and
\begin{equation}\label{hatw1}
\| D^2_\xi w_{i+1,\e} \|_{\ve , N-2 , \sigma } + \|
D_\xi w_{i+1,\e} \|_{\ve , N-3 , \sigma } +\|w_{i+1,\e} \|_{\ve, N-4
, \sigma}\le C \ve^{i+1}
\end{equation}
and, for any integer $\ell$
$$
\|\nabla^{(\ell)}_{y} w_{i+1,\e}(y,\cdot)\|_{\ve,N-2,\sigma} \leq \beta C_l \ve^{i+1}
\qquad \quad y = \rho z \in K
$$
for any $i=0, 1, 2, 3$.

Thanks to the above decomposition (\ref{decomp}), we have the validity
of the following expansion for $E(\phi)$.

\begin{eqnarray}\label{expane1}
&& E ({\delta \over \, \mu_\e } {\tilde{\mathcal T}}_{\, \mu_\e ,
\, d_\e} (Z_{N+1} ) \bar \chi_\e )
 =\rho^{-k}\e \frac{1}{2}    \int_K\left[ A_{1,\e} \e^{1+\frac{2}{N-2}}
|\nabla_K ( \delta (1+ o(\e  ) \beta_1^\e (y)  )) |^2  -(N-2)A_1 \frac{\mu_0^{N-4}}{(d_{N}^0)^{N-2}} \delta^2 \right.\nonumber\\
&&\qquad  \left.+(N-2)A_1 \frac{\mu_0^{N-3}}{(d_{N}^0)^{N-1}}\delta d_N +\e^{\frac{1}{N-2}}{\delta \over \, \mu_0 }\left(\frac{\mu_0}{d_{N}^0}\right)^{N-1}g_{N+1}\left(\frac{\mu_0}{d_{N}^0}\right)\right]dz
\end{eqnarray}
\begin{eqnarray}\label{expane1bn}
&& E ({d_N \over \, \mu_\e } {\tilde{\mathcal T}}_{\, \mu_\e ,
\, d_\e} (Z_{N} ) \bar \chi_\e )
 =\rho^{-k}\rho^2\frac{1}{2}    \int_K\left[ A_{2,\e} \e
|\nabla_K ( d_N (1+ o(\e  ) \beta_2^\e (y)  )) |^2 - (N-2)A_1 \frac{\mu_0^{N-3}}{(d_{N}^0)^{N-1}}\delta d_N \right.\nonumber\\
&&\qquad  \left.+(N-1)A_3 \frac{\mu_0^{N-2}}{(d_{N}^0)^{N}}d_N^2 +\e^{\frac{1}{N-2}}{d_N \over \, \mu_0 }\left(\frac{\mu_0}{d_{N}^0}\right)^{N}g_{N}\left(\frac{\mu_0}{d_{N}^0}\right)\right]dz
\end{eqnarray}
\begin{eqnarray}\label{expane1bj}
&& E ({d_j \over \, \mu_\e } {\tilde{\mathcal T}}_{\, \mu_\e ,
\, d_\e} (Z_{j} ) \bar \chi_\e ) \\
 &=&\rho^{-k}\rho^2\frac{1}{2}    \int_K\left[ A_{3,\e}
|\nabla_K ( d_j (1+ o(\e  ) \beta_3^\e (y)  )) |^2- {{\mathcal R}_{mj} \over 4}  d_j d_m +\e^{\frac{1}{N-2}}{d_j \over \, \mu_0 }\left(\frac{\mu_0}{d_{N}^0}\right)^{N-1}g_{j}\left(\frac{\mu_0}{d_{N}^0}\right)\right]dz\nonumber
\end{eqnarray}

\begin{eqnarray}\label{tere3e}
&&E({e\over \, \mu_\e} {\tilde{\mathcal T}}_{\, \mu_\e ,
\, d_\e} (Z_0)) \\
&=& \rho^{-k}\frac{1}{2}\int_K\left[D_1\,|\partial_{ a} e  + e^{-\sqrt{\la_1}
\ve^{-\gamma}} \beta_4^\e (y) e |^2- \la_1  D_1  e^2- D_2d_{N}^0 e\right]\bigg(1+ \e
O(e^{-\sqrt{\la_1}|\xi|})\bigg). \nonumber
\end{eqnarray}

\noindent
Therefore, $\mu$ and $d_1,\cdots,d_{N-1},d_N$ and $e$ satisfy
\begin{eqnarray}\label{mudidne}
\left\{
\begin{array}{llll}
             L_{N+1}(\d,d_N) := -c_1 \e^{1+\frac{2}{N-2}}\mu_0 \Delta_K \delta+A \delta +B  d_N=\alpha_{N+1}+\e M_{N+1};\\[3mm]
             L_{N}(\d,d_N) := -c_2 \e \mu_0 \Delta_K d_N +B \delta  +C  d_N=\alpha_{N}+\e M_{N}; \\[3mm]
             L_{j}(\bar d) :=   - \Delta_K
 d_j+\bigg( \tilde g^{ab} R_{mabj}- \G_a^c(E_m) \G_c^a(E_j)  \bigg)  d_m=\alpha_{j}+\e M_{j},\  j=1,\ldots,N-1;\\[3mm]
             L_{0}(e) :=    \Delta_K e +D_1\la_1e+D_2d_N=\alpha_{0}+\e Q_0+\e^2 M_{0},
\end{array}
\right.
\end{eqnarray}
where
$$
A=-(N-2)A_1 \frac{\mu_0^{N-3}}{(d_{N}^0)^{N-2}} ,\quad
B=(N-2)A_1 \frac{\mu_0^{N-2}}{(d_{N}^0)^{N-1}},\quad
C=-(N-1)A_3 \frac{\mu_0^{N-1}}{(d_{N}^0)^{N}},
$$
with $AC-B^2>0$, and
$$
D_1 = \int_{\R^N} Z_0^2(\xi)\,d\xi, \quad \hbox{and }\qquad
D_2=2H_{jj}d_{N,\ve}^0\int_{\R^N}\partial^2_{jj}UZ_0d\xi.
$$
The functions $\alpha_j$ are explicit function of $z$ in $K$,smooth and
uniformly bounded in $\e$.   The operators $M_i=M_i(\mu,d,e)$ can be decomposed in the following form
$$
M_i(\mu,d,e)=A_i(\mu,d,e)+K_i(\mu,d,e)
$$
where $K_i$ is uniformly bounded in $L^{\infty}(K)$ for $(\mu,d,e)$ and is also compact. The operator $A_i$ depends on $(\mu,d,e)$ and their first and second derivatives and it is Lipschitz in this region, that is
$$
\|A_i(\mu_1,d_1,e_1)-A_i(\mu_2,d_2,e_2)\|_\infty\leq Co(1)\|(\mu_1-\mu_2,d_1-d_2,e_1-e_2)\|.
$$

We remark that the dependence on $\ddot \mu$, $\ddot d$ and $\ddot
e$ is linear. Finally, the operator $Q_0$ is quadratic in $d$ and it
is uniformly bounded in $L^\infty (K)$ for $(\d , d , e)$
satisfying \eqref{bf2}-\eqref{bf45}

\medskip
Our goal is now to solve \equ{mudidne} in $\d$, $d$ and $e$. To
do so, we first analyze the invertibility of the linear operators
$L_i$.

We start with a linear theory in $L^\infty$ setting  for the problem
\begin{equation}
\label{nn} L_{N+1} (\d,d_N) = h_1 , \quad L_N (\d,d_N) = h_2,
\end{equation}
with $h_1$ and $h_2$ bounded. Arguing as in  the proof of Lemma 8.1 in \cite{dmpa} (with obvious modifications), we can prove that, assuming  $A<0$, $C <0$ and $AC - B^2
>0$ and that
$\| h_1\|_\infty + \| h_2  \|_\infty $ is bounded. Then there exist
$(\mu, d)$ solution  to the above system and a
constant $C$ such that
\begin{equation}\label{eq:estdeltadN}
\| \mu\|_\infty + \| d_N \|_\infty + \ve^{ \frac 12 + \frac 1{N-2}}\|
\nabla_K\mu\|_\infty + \ve^{\frac 12} \| \nabla_K d_N \|_\infty \le
C\,[\,\| h_1\|_\infty + \| h_2  \|_\infty \, ]\ .
\end{equation}
As we mentioned above, to abtain this we follows the lines of the proof of Lemma 8.1 in \cite{dmpa}.  For existence we use the fact that the system \equ{nn} has a variational structure with  associated energy functional
$$
J(\delta,d_N)=\frac{1}{2}c_1 \e^{1+\frac{2}{N-2}}\mu_0 \int_K|\nabla_K \delta|^2+c_2 \e \mu_0 \int_K|\nabla_K d_N|^2+\frac{1}{2}\big(A\int_K \delta^2 +2B \int_K \delta d_N+
C\int_K d_N^2 \big)
$$
and clearly by our assumption on the constants $A,B,C$ this energy functional
 is positive, bounded from below away from zero and convex. Then,
existence of solution   follows.
The a-priori estimate \eqref{eq:estdeltadN} follows by a contradiction argument (as in Lemma 8.1 in \cite{dmpa}). Indeed, if \eqref{eq:estdeltadN} is false, we have existence of a
sequence $(h_{1n}, h_{2n})$ with
$
\| h_{1n}\|_\infty + \| h_{2n}  \|_\infty \to 0 \ ,
$
and a sequence of solutions $(\d_n , (d_N)_n )$ with
$$
\| \d_n \|_\infty + \| (d_N)_n  \|_\infty + \ve^{ \frac 12 + \frac
1{N-2}}\| \nabla_K \d_n  \|_\infty + \ve^{\frac 12} \| \nabla_K (d_N)_n
\|_\infty = 1.
$$
Since $A <0$ and $C <0$ and $C - {{B^2 } \over A} >0$ and  applying the maximum principle, Ascoli-Arzel\'a theorem we end up with a contradiction.
Now for every $ j=1,\ldots,N-1$ the operator $L_j$ is invertible by the non degeneracy of the submanifold $K$.  We can then prove that the equation
$
L_j \bar d=f
$
is solvable on $\bar d$ and the following estimate holds true
\begin{equation}\label{eq:estbard}
\| \bar d\|_\infty +  \|
\partial_a\bar d\|_\infty + \| \partial^2_{ab} \bar d \|_\infty \le
C\,\| f\|_\infty \ .
\end{equation}
We are then left with the study of the invertibility of the operator $L_0$. we prove it as the following result.
\end{proof}

\begin{lemma}\label{prop-gap}
There is a sequence $\e=\varepsilon_j\searrow0$ such that for any $\varphi\in C^{0,\alpha}(K)$, there exists a unique $e\in C^{2,\alpha}(K)$ such that
\begin{equation}\label{eq-K-e}
L_0(e)=\varphi
\end{equation}
with the property
\begin{equation}\label{estimate-e}
\|e\|_*
:=\|e\|_{L^\infty(K)}+\rho\|\nabla e\|_{L^\infty(K)}+\rho^2\|\nabla^2e\|_{L^\infty(K)}
\leq C\rho^{-k}\|\varphi\|_{L^\infty(K)},
\end{equation}
where $C$ is a positive constant independent of $\varepsilon_j$.
\end{lemma}

\begin{proof}
The proof is classical, the arguments are  similar in spirit to the ones used in \cite{demamu}, \cite{mmah} and some references therein. We also refer the reader to the papers \cite{mmp,mazzeopacard} for a different setting.  So we will omit the proof here.
\end{proof}

\noindent{\it Proof of Proposition \ref{glinear}:}
Using Proposition \ref{teouffa}, we can get the existence of solutions to the linear problem in the whole domain
$\Omega_\rho$, we refer the reader to \cite{demamu} for the detail proof.

\section{Appendix A}\label{secapp}

\setcounter{equation}{0}

\noindent{\bf Proofs of \equ{proh1zn1}-\equ{proh1z0h}:}
We recall that
$
h_{1,\e}   = h_{11}+\ve h_{12}+\rho h_{13}
$
where we have set
\begin{eqnarray*}\label{g1e}
h_{11} & =& pU^{p-1}\bar{U} +\ve U^p\log U, \\[3mm]
h_{12}&=&-\frac{N-2}{2}U^p \log(\mu_{0,\ve})-2d^0_{N,\ve}H_{ij}\partial_{ij}^2U-\lambda_1 e_{0,\ve} Z_0,\\[3mm]
h_{13}&=& \mu_{0,\ve}  \left[-2\xi_NH_{ij}\partial_{ij}^2 U+H_{\alpha\alpha}\partial_NU\right].
\end{eqnarray*}
By the result of \cite{dmpa}, we have
\begin{equation}\label{proh1zn1a}
\int_{\hat\DD}h_{11}Z_{N+1}d\xi  =  \ve \left[-
A_1 \left( { \mu_{0,\ve} \over   d_{N,\ve}^0} \right)^{N-2} +  A_2 +
\ve^{1\over N-2}
  \left({\mu_{0,\ve} \over   d_{N,\ve}^0} \right)^{N-1} g_{N+1} \left( {\mu_{0,\ve} \over   d_{N,\ve}^0} \right) \right]
\end{equation}
\begin{eqnarray}\label{proh1zna}
\int_{\hat\DD}h_{11}Z_{N}d\xi = \ve^{1+\frac{1}{N-2}} \left[A_3
  \left( { \mu_{0,\ve} \over   d_{N,\ve}^0} \right)^{N-1}   +
\ve^{1\over N-2}
  \left({\mu_{0,\ve} \over   d_{N,\ve}^0} \right)^{N} g_{N} \left( {\mu_{0,\ve} \over   d_{N,\ve}^0} \right)\right]
\end{eqnarray}
\begin{equation}\label{proh1zla}
\int_{\hat\DD}h_{11}Z_{l}d\xi = \ve^{2+\frac{3}{N-2}} g_{l}\left( {\mu_{0,\ve} \over   d_{N,\ve}^0} \right) \qquad \mbox{for}\ \ l=1,\ldots,N-1,
\end{equation}
\begin{eqnarray}\label{proh1z0ha}
\int_{\hat\DD}h_{11}Z_{0}d\xi = \ve  \left[A_4
  \left( { \mu_{0,\ve} \over   d_{N,\ve}^0} \right)^{N-2} +  A_5+
\ve^{1\over N-2}
  \left({\mu_{0,\ve} \over   d_{N,\ve}^0} \right)^{N-1} g_{0}\left( {\mu_{0,\ve} \over   d_{N,\ve}^0} \right)\right]
\end{eqnarray}
where the functions $g_i$ are smooth function
with $g_i (0) \not= 0$ and $A_i$ are positive constants. In particular,
$
A_3={p \alpha_N^{N+2 \over 2} (N-2)^2 \over 2^{N-1}} ( \int {\xi_N^2 \over
(1+ |\xi|^2)^{N+4 \over 2} })d\xi.
$

It remain to compute $h_{12}$ and $h_{13}$ product  with $Z_i$
for $i = 0, 1, \dots, N+1$. First, by symmetry, we have
\begin{equation}\label{proh1zlab}
\int_{\hat\DD}(\ve h_{12}+\rho h_{13})Z_{l}d\xi = \ve^{2+\frac{3}{N-2}} \Theta \qquad \mbox{for}\ \ l=1,\ldots,N-1,
\end{equation}
where $\Theta$ denotes a sum of functions of the form
\begin{eqnarray}\label{youjie}
&&f_1(\rho z) \left[ f_2 ( \mu_{0,\ve}, d^0_{N,\ve},e_{0,\ve}, \partial_a\mu_{0,\ve}, \partial_ad^0_{N,\ve},\partial_e e_{0,\ve}) +\right.\nonumber\\[3mm]
 &&\qquad \left.+o(1) f_3 ( \mu_{0,\ve}, d^0_{N,\ve},e_{0,\ve},  \partial_a\mu_{0,\ve}, \partial_ad^0_{N,\ve}, \partial_ae_{0,\ve} , \partial^2_{aa}\mu_{0,\ve},
\partial^2_{aa}d^0_{N,\ve}  , \partial^2_{aa}e_{0,\ve} ) \right]
\end{eqnarray}
where $f_1$ is a smooth function uniformly bounded in $\ve$, $f_2$ and $f_3$ are smooth functions of their arguments, uniformly bounded in $\ve$ as $\mu_{0,\ve}$, $d^0_{N,\ve}$ and $e_{0,\ve}$ are uniformly bounded, and $o(1)\to0$ as $\ve\to0$.

\medskip

First, product with $Z_{N+1}$, we have
\begin{eqnarray*}
&&\int_{\hat\DD}(\ve h_{12}+\rho h_{13})Z_{N+1}d\xi= \ve\int_{\hat\DD}h_{12}Z_{N+1}d\xi+\ve^2\Theta\nonumber\\
&=&   \ve\int_{\hat\DD}\left\{-\frac{N-2}{2}U^p \log(\mu_{0,\ve}) -2d^0_{N,\ve}H_{ij}\partial_{ij}^2U\right\}Z_{N+1}d\xi+\ve^2\Theta
\end{eqnarray*}
where $\Theta$ is a sum of functions of the form (\ref{youjie}).

We set
$
U_\lambda(\xi)=\alpha_N
\left(\frac{\lambda}{\lambda^2+|\xi|^2}\right)^\frac{N-2}{2}.
$
Since $( \pa_\lambda U_\lambda )_{|\lambda=1} =
-Z_{N+1}$, we have
$$
\int_{\hat{\R^N}} U^{p} Z_{N+1}=\int_{ \R^N } U^{\frac{N+2}{N-2}} Z_{N+1} =
-\frac{N-2}{2N} \pa_\lambda \left( \int_{ \R^N }
U_\lambda^\frac{2N}{N-2} \right)_{|\lambda=1}=0.
$$
Here we used the fact that  and $ \int_{\R^N} U_\lambda^\frac{2N}{N-2}$ does not depend
on $\lambda$ (by simple change of variables argument).  Moreover,
$$
 H_{ij}\int_{\R^N}\partial_{ij}^2U Z_{N+1}d\xi= H_{11}\int_{\R^N}\partial_{11}^2U Z_{N+1}d\xi =H_{11}\int_{\R^N}\partial_{11}^2U (\gamma U+\xi_l\partial_l U)d\xi
$$
$$
 = -H_{11}\gamma\int_{\R^N}|\partial_{ 1}^2U|^2d\xi+ \frac{1}{N}H_{11} \int_{\R^N}\Delta  U \xi_l\partial_l U d\xi
$$
$$
 = -H_{11}\gamma\int_{\R^N}|\partial_{ 1}^2U|^2d\xi- \frac{1}{N}H_{11} \int_{\R^N}  U^{p} \xi_l\partial_l U d\xi
$$
$$
 = -H_{11}\gamma\int_{\R^N}|\partial_{ 1}^2U|^2d\xi- \frac{1}{N(p+1)}H_{11} \int_{\R^N}    \partial_l (U^{p+1})\xi_l d\xi
$$
$$
= -H_{11}\gamma\int_{\R^N}|\partial_{ 1}^2U|^2d\xi+ \frac{1}{ (p+1)}H_{11} \int_{\R^N}   U^{p+1}  d\xi
$$
$$
 = -H_{11}\gamma\int_{\R^N}|\partial_{ 1}^2U|^2d\xi+ \frac{N}{ (p+1)}H_{11} \int_{\R^N}  |\partial_{ 1} U|^2  d\xi=0.
$$
Collecting these facts, we get
$
 \int_{\hat\DD}(\ve h_{12}+\rho h_{13})Z_{N+1}d\xi
=    \ve^{2} \Theta,
$
where $\Theta$ is a sum of functions of the form (\ref{youjie}).

Next,
product with $Z_N$, we have
\begin{eqnarray*}
&&\int_{\hat\DD}(\ve h_{12}+\rho h_{13})Z_Nd\xi=\rho\int_{\hat\DD}h_{13}Z_Nd\xi+\ve^{2}\Theta\nonumber\\
&=& \rho \mu_{0,\ve}\left[-\int_{\R^N}2\xi_NH_{ij}\partial_{ij}^2U\partial_NUd\xi+H_{\alpha\alpha}\int_{\R^N}|\partial_NU|^2d\xi\right]
+\ve^2\Theta\nonumber\\
&=& \rho \mu_{0,\ve}\left[- (H_{jj}-H_{\alpha\alpha}) \int_{\R^N}|\partial_NU|^2d\xi\right]
+\ve^2\Theta\nonumber\\
&=&  \rho \mu_{0,\ve}  H_{aa} \int_{\R^N}|\partial_NU|^2d\xi
+\ve^2\Theta,
\end{eqnarray*}
where $\Theta$ is a sum of functions of the form (\ref{youjie}).
Here we used the following fact
\begin{eqnarray*}
&& \int_{\R^N} \xi_NH_{ij}\partial_{ij}^2U\partial_NUd\xi = H_{jj}\int_{\R^N} \xi_N\partial_{jj}^2U\partial_NUd\xi\nonumber\\
 &=&\frac{1}{N-1}H_{jj}\int_{\R^N} \xi_N\partial_NU\sum\limits_{i=1}^{N-1}\partial_{ii}^2Ud\xi\nonumber\\
 &=&\frac{1}{N-1}H_{jj}\int_{\R^N} \xi_N\partial_NU(\Delta U-\partial_{NN}^2U)d\xi\nonumber\\
 &=&-\frac{1}{N-1}H_{jj}\int_{\R^N} \xi_N\partial_NU(  U^p+\partial_{NN}^2U)d\xi\nonumber\\
 &=&-\frac{1}{N-1}H_{jj}\left[\frac{1}{p+1}\int_{\R^N} \xi_N\partial_N(U^{p+1} )d\xi+\frac{1}{2}\int_{\R^N} \xi_N\partial_N(|\partial_N U|^2 )d\xi\right]\nonumber\\
 &=& \frac{1}{N-1}H_{jj}\left[\frac{1}{p+1}\int_{\R^N}   U^{p+1}  d\xi+\frac{1}{2}\int_{\R^N}  |\partial_N U|^2  d\xi\right]\nonumber\\
 &=& \frac{1}{N-1}H_{jj}\left[\frac{1}{p+1}N \int_{\R^N}  |\partial_N U|^2  d\xi+\frac{1}{2}\int_{\R^N}  |\partial_N U|^2  d\xi\right]\nonumber\\
 &=& \frac{1}{2}H_{jj}  \int_{\R^N}  |\partial_N U|^2  d\xi,
\end{eqnarray*}
since $\int_{\R^N}   U^{p+1}  d\xi=\int_{\R^N}  (-\Delta U)U  d\xi=\int_{\R^N}   |\nabla U|^2 d\xi=N\int_{\R^N}  |\partial_N U|^2  d\xi$.

\medskip

Finally, using the orthogonality in $L^2$ of $Z_0$ with respect to $Z_i$, for $i=1, \ldots , N+1$,
 direct computations show
$$
\int_{\hat\DD}(\ve h_{12}+\rho h_{13})Z_{0}d\xi= -A_7\log(\mu_{0,\ve})-\lambda_1 e_{0,\ve}-2H_{jj}d_{N,\ve}^0\int_{\R^N}\partial^2_{jj}UZ_0d\xi+\ve^2\Theta
$$
where $\Theta$ is a sum of functions of the form (\ref{youjie}), and $A_7=\frac{N-2}{2}\int_{\mathbb{R}^N}U^pZ_0d\xi$.

\medskip

Collecting all formulas from  (\ref{proh1zn1a}), we get the results.

\section{Acknowledgments}

The research of the first  author
has been partly supported by  National Natural Science Foundation of China 11501469. The research of the second  author has been  supported by  Fondecyt Grant 1140311, fondo Basal CMM and  ``Millennium Nucleus Center for Analysis of PDE NC130017".  The research of the third  author has been partly supported by Fondecyt Grant 1160135 and  ``Millennium Nucleus Center for Analysis of PDE NC130017".

\end{document}